%% file: main.tex
\documentclass{amsart}

\usepackage{url,bm}
\usepackage[curve]{xypic}
\usepackage[mathscr]{euscript}
\usepackage{epsfig,graphicx,color}

\newtheorem{theorem}{\bf Theorem}[section]
\newtheorem{proposition}[theorem]{\bf Proposition}

\newtheorem{lemma}[theorem]{\bf Lemma}
\newtheorem{corollary}[theorem]{\bf Corollary}

\newtheorem{example}[theorem]{\bf Example}
\newtheorem{remark}[theorem]{\bf Remark}
\newtheorem{question}{\bf Question}

\renewcommand{\tilde}{\widetilde}

\newcommand{\cupprod}{\mathbin{\smile}}
\def\PD{{\rm PD}}

\def\pr{{\rm pr}}

\def\C{{\mathbb C}}
\def\D{{\mathcal D}}

\def\0{{\mathbb 0}}
\def\P{{{\mathbb P}}}

\def\Z{{\mathbb Z}}

\def\AA{{\mathcal A}}
\def\BB{{\mathcal B}}

\def\cal{\mathcal}

\def\int{\mathrm{inter}}

\def\To_#1{\mathrel{\mathop{\longrightarrow}\limits_{#1}}}

\title{Computing dynamical degrees of rational maps on moduli space}

\begin{author}[S. Koch]{Sarah Koch}\thanks{The research of the first author was supported in part by the NSF DMS-1300315 and the Alfred P. Sloan Foundation}
\email{kochsc$@$umich.edu}
\address{Department of Mathematics\\
University of Michigan\\
East Hall, 530 Church Street\\
Ann Arbor MI 48109\\
United States }
\end{author}
\begin{author}[R. K. W. Roeder]{Roland K. W. Roeder}\thanks{The research of the second author was supported in part by NSF grant DMS-1102597 and startup funds from the Department of Mathematics at IUPUI}
\email{rroeder$@$math.iupui.edu}
\address{ %
IUPUI Department of Mathematical Sciences\\
LD Building, Room 224Q\\
402 North Blackford Street\\
Indianapolis, Indiana 46202-3267\\
 United States }
\end{author}

\begin{document}

\maketitle
\begin{abstract}
The dynamical degrees of a rational map $f:X\dashrightarrow X$ are fundamental
invariants describing the rate of growth of the action of iterates of $f$ on
the cohomology of $X$. When $f$ has nonempty indeterminacy set, these
quantities can be very difficult to determine. We study rational maps
$f:X^N\dashrightarrow X^N$, where $X^N$ is isomorphic to the Deligne-Mumford
compactification $\overline {\mathcal M}_{0,N+3}$.  We exploit the stratified
structure of $X^N$ to provide new examples of rational maps, in arbitrary
dimension, for which the action on cohomology behaves functorially under iteration.  From this,
all dynamical degrees can be readily computed (given enough book-keeping and
computing time). In this article, we explicitly compute {\em all} of the dynamical
degrees for all such maps $f:X^N\dashrightarrow X^N$, where
$\mathrm{dim}(X^N)\leq 3$ and the first dynamical degrees for the mappings where $\mathrm{dim}(X^N)\leq 5$.  These examples naturally arise in the setting of
Thurston's topological characterization of rational maps.
\end{abstract}

\input{intro.tex}

\input{action_on_cohomology.tex}

\input{DM.tex}

\input{maps.tex}

\input{stability.tex}

\input{dynamical_degrees.tex}

\input{data.tex}

\input{bibliography}
\end{document}

%% file: intro.tex
\section{Introduction}\label{SEC:INTRO}

%
%Each map in the family that we will study can be decomposed as a composition
%$\alpha \circ h: \DMC \dashrightarrow \DMC$, where $h: \DMC \dashrightarrow \DMC$ is a
%relatively simple map to understand, and $\alpha\in\mathrm{Aut}(\DMC)$. This
%observation greatly simplifies the analysis of the dynamics, and provides a
%combinatorial tool with which we understand each of our maps. In particular,
%since each of our maps $\DMC \dashrightarrow \DMC$ is algebraically stable, we can
%readily compute the {\em{first dynamical degree}}, $\lambda_1$. To this end, we
%exploit the combinatorics of this decomposition to systematically compute the
%minimal polynomials for $\lambda_1$. 

% A few things: The Y space is really Wbar. The fact that the map that we need for the Amerik lemma is 1) holo and 2) finite comes from the fact that the map W-->M_B is a finite cover. 

Let $X$ be a smooth, compact, complex algebraic variety of dimension $N$. A rational map
$f:X\dashrightarrow X$ induces a pullback action $f^*:H^{k,k}(X;\C)\to
H^{k,k}(X;\C$) (defined in Section \ref{action_on_cohomology}). 
A typical starting point for studying the dynamics associated to iterating 
$f$ is to compute the {\em dynamical degrees}
\begin{eqnarray}\label{EQN:DEF_DYN_DEG}
\lambda_k(f) := \lim_{n \rightarrow \infty} \|(f^n)^*:H^{k,k}(X;\C) \rightarrow H^{k,k}(X;\C)\|^{1/n},
\end{eqnarray}
which are defined for $0 \leq k \leq N$.
Given the dynamical degrees of $f$, there is a precise description of what ergodic properties
$f$ {\em should} have; see, for example, \cite{GUEDJ_ERGODIC}.
These properties have
been established when  $\lambda_N(f)$ is maximal \cite{GUEDJ,DNT} or when $\dim(X) =
2$, $\lambda_1(f) > \lambda_2(f)$, and certain minor technical hypotheses are
satisfied~\cite{DDG}.

Dynamical degrees were originally introduced by
Friedland \cite{FRIEDLAND} and later by Russakovskii and Shiffman
\cite{RUSS_SHIFF} and shown to be invariant under birational conjugacy by Dinh
and Sibony \cite{DS_BOUND}.  Dynamical degrees were originally defined with a ${\rm lim sup}$ instead of a limit in line (\ref{EQN:DEF_DYN_DEG}) above; however, it was shown in \cite{DS_BOUND,DS_REGULARIZATION} that the limit always exists.

If the map $f$ has points of
indeterminacy, then the iterates of $f$ may not act functorially on $H^{k,k}(X;\C)$, which can be a formidable obstacle to computing the dynamical
degree $\lambda_k(f)$. If the action of $f^\ast$ on $H^{k,k}(X;\C)$ is functorial; that is, for all
$m>0$
\begin{eqnarray*}
(f^\ast)^m:H^{k,k}(X;\C)\to H^{k,k}(X;\C)  \quad\text{equals}\quad  (f^m)^\ast : H^{k,k}(X;\C)\to H^{k,k}(X;\C)
\end{eqnarray*}
then the map $f:X\dashrightarrow X$ said to be {\em{$k$-stable}}. In this case,
it immediately follows that the dynamical degree $\lambda_k(f)$ is the spectral
radius of  $f^\ast: H^{k,k}(X;\C)\to H^{k,k}(X;\C)$. If $f:X\dashrightarrow X$
is $k$-stable for all $1\leq k \leq N$, then the map $f:X\dashrightarrow X$ is
said to be {\em{algebraically stable}}. Note that $f^*$ is automatically functorial
on $H^{N,N}(X;\C)$ and $\lambda_N(f)$ is the topological degree of $f$. For
more background and discussion of dynamical degrees and algebraic stability, we
refer the reader to \cite{BEDFORD_DYN_DEG,ROEDER_FUNC}. 

Given an arbitrary map $f: X \dashrightarrow X$, the problem of verifying that
$f$ is algebraically stable (or modifying $X$ in order to conjugate $f$ to an
algebraically stable map) can be quite subtle, as is the problem of determining
all of the dynamical degrees $\lambda_1(f),\lambda_2(f),\ldots,\lambda_N(f)$.
The purpose of this article is to study these problems for a specific family of
maps $f_\rho:X^N \dashrightarrow  X^N$ where both the map $f_\rho$ and the
space $X^N$ have additional structure.

More specifically, the space $X^N$ will be isomorphic to
${\overline{\cal M}}_{0,n}$, the Deligne-Mumford compactification of ${\cal
M}_{0,n}$, where ${\cal M}_{0,n}$ is the moduli space of genus $0$ curves with
$n$ labeled points. The space ${\overline{\cal M}}_{0,n}$ is a smooth
projective variety of dimension $N=n-3$ \cite{knudsen,knudsenmumford}. Given a
permutation $\rho\in S_n$, we build a map $f_\rho:=g_\rho \circ
s:X^N\dashrightarrow X^N$, where $ s:X^N\dashrightarrow X^N$ is a relatively
simple map to understand (although it has a nonempty indeterminacy set), and
$g_\rho:X^N\to X^N$ is an automorphism of $X^N$ induced by the permutation~$\rho$. 
The resulting $f_\rho$ has topological degree $\lambda_N(f_\rho) = 2^N$, so it remains to consider the other dynamical degrees $\lambda_k(f_\rho)$ for $1 \leq k < N$.

We can prove stability is a somewhat wider context:
For a fixed $N \geq 1$, let $\mathcal{F}^N$ denote the semi-group of rational self-mappings of $X^N$
generated under composition by all of the mappings of the form $f_\rho: X^N \dashrightarrow X^N$.
Our first main result is:

\begin{theorem}\label{THM_FUNCTORIAL_SEMIGROUP}
For any fixed $N \geq 1$ the semi-group $\mathcal{F}^N$ acts functorially on all cohomology groups of
$X^N$, i.e. for any $f_1,f_2 \in \mathcal{F}^N$ and any $1 \leq k \leq N$ we have
\begin{eqnarray*}
(f_1\circ f_2)^\ast = f_2^\ast \circ f_1^\ast: H^{k,k}(X^N;\C)\to H^{k,k}(X^N;\C).
\end{eqnarray*}
\end{theorem}

An immediate corollary is:

\begin{corollary}\label{algstabthm}
For any $N \geq 1$ any
$f \in \mathcal{F}^N$ is an algebraically stable self-map of~$X^N$.  In particular,
for any $1 \leq k \leq N$ we have that $\lambda_k(f)$ equals the spectral radius of
\[
f^\ast:H^{k,k}(X^N;\C)\to H^{k,k}(X^N;\C).
\]
\end{corollary}

For the remainder of the paper we focus on computing the dynamical degrees of the generators $f_\rho$ of $\mathcal{F}^N$.  By Corollary \ref{algstabthm}
they {\em should} be easy to compute.
However, we are confronted with another challenge: the dimension
of $H^{k,k}(X^N;\C)$ grows exponentially with $N$ (see Table  \ref{TABLE_COH_DIMS}).  These numbers were computed using a
theorem of S. Keel (Theorem \ref{keelthm} below, published in \cite{KEEL}) which provides
generators and relations for the cohomology ring $H^*(X^N;\C)$.  Keel's Theorem will play
a central role in all of our calculations. Computing $(f_\rho)^\ast$ on $H^{k,k}(X^N;\C)$ is rather difficult as because $\mathrm{dim}(H^{k,k}(X^n;\C))$ is large. 
%Because of the large dimensions of $H^{k,k}(X^n;\C)$, 
%the calculation of $(f_\rho)^\ast$ is quite difficult.  

Our second main result is that 
we provide an algorithmic approach to computing
$(f_\rho)^*: H^{1,1}(X;\C) \rightarrow H^{1,1}(X;\C)$, which is presented in  Section \ref{SEC:DYN_DEG}.  This allows us to
readily compute $\lambda_1(f_\rho)$ for any $N$ and $\rho$, using the computer algebra system Sage
\cite{SAGE}. 
% The results are expressed by giving a
%minimal polynomial for $\lambda_1(f_\rho)$ and a decimal approximation.  
Values
of $\lambda_1(f_\rho)$ for $N = 2,3,4,5$ and various $\rho$ are tabulated in
Section \ref{SEC:CATALOG}.

\begin{table}
\begin{center}
\begin{tabular}{|c|ccccccc|}
\hline
$N \setminus k$ & 0 & 1 & 2 & 3 & 4 & 5 & 6  \\
0  &      1  &      & & & & &   \\
1  &    1    &   1   &     & & & &   \\
2  &    1    &   5   &    1  &       & & &   \\
3  &    1    &   16  &   16  &    1   &      & &   \\
4  &    1    &   42  &   127 &    42  &   1   &     &   \\
5  &    1    &  99   &   715 &    715 &    99  &    1   &         \\
6  &    1    &   219 &    3292 &    7723 &   3292 &   219 &    1   \\
\hline
\end{tabular}
\vspace{0.05in}
\caption{\label{TABLE_COH_DIMS}Table of dimensions of $H^{k,k}(X^N)$.}
\end{center}
\end{table}

It is far more technical to compute $(f_\rho)^\ast$ for $k \geq 2$ because
a subvariety $V \subset X^N$ of codimension greater than or equal to $2$ may have
preimage $f^{-1}(V)$ lying entirely in the indeterminacy set $I(f_\rho)$. 
Our final main result is 
computation of  
\begin{eqnarray*}
(f_\rho)^*: H^{2,2}(X^3;\C) \rightarrow H^{2,2}(X^3;\C)
\end{eqnarray*}
in Section \ref{SEC:DYN_DEG}.
The
resulting values for $\lambda_2(f_\rho)$ tabulated in Section
\ref{SEC:CATALOG}.  With sufficient book-keeping, we expect that this can be
done for all $N$ and $k$.

\begin{remark}For a given $\rho$, the space $X^N$ may not be optimal, meaning
that there is a space $Z^N$, obtained by blowing down certain hypersurfaces, on
which (a conjugate of) $f_\rho$ is still algebraically stable.  (For example, certain choices of $\rho$, including $\rho = {\rm id}$,
result in a mapping 
$f_\rho$ that is algebraically stable on $\P^N$.)
Similarly, if one is only interested in
$k$-stability for a particular value of $k$, there may be a blow down $Z^N$ of
$X^N$ on which all of the mappings $f_\rho$ are $k$-stable (see Remark \ref{REM:ONE_STABILITY}).  The merit of
working with $X^N$ is that every mapping $f_\rho$ is $k$-stable for all $1 \leq
k \leq N$ on the same space $X^N$.  \end{remark}

% other degrees - more than just squaring?
% previous work? when is it done in arb degrees?

\noindent{\bf In the literature.} Dynamical degrees have been extensively
studied for maps $f:X\dashrightarrow X$ where $X$ is a surface. If $f$ is a
bimeromorphic map of a compact K\"ahler surface, J. Diller and C.
Favre~\cite{DILLER_FAVRE} proved that there is a proper modification
$\pi:\widehat X \to X$ so that $f$ lifts to an algebraically stable map
$\widehat f:\widehat X\dashrightarrow \widehat X$.  The space $\widehat X$ and lifted
map $\widehat f$ are called a {\em stabilization} of $f$.  However in~\cite{FAVRE}, C.
Favre found examples of monomial maps $f:\P^2\dashrightarrow\P^2$ of topological degree $\geq 2$ for which no
such stabilization exists. 

In the higher dimensional case $f:X\dashrightarrow X$, the question of the
functoriality of $f^*$ on $H^{1,1}(X;\C)$ (that is, whether or not $f$ is
$1$-stable) has been extensively studied \cite{BEDFORD_KIM, BEDFORD_KIM2,
BEDFORD_TRUONG, JONSSON_WULCAN, HASSELBLATT_PROPP}. The
functoriality of $f^*$ on $H^{k,k}(X;\C)$ for $2\leq k\leq N-1$ is typically even more delicate.
In \cite{LIN1}, Lin computes all of the dynamical degrees for
monomial maps $\P^3\dashrightarrow \P^3$. In \cite{BK,BCK}, Bedford-Kim and
then Bedford-Cantat-Kim study pseudo-automorphisms of 3-dimensional
manifolds, computing all dynamical degrees for a certain family of such maps.
In \cite{AMERIK}, Amerik computes all dynamical degrees for a particular map
$f:X\dashrightarrow X$, where $X$ is a 4-dimensional smooth compact complex projective variety arising
in a algebro-geometric context. In \cite{FAVRE_WULCAN,LIN2}, Favre-Wulcan and Lin
compute all dynamical degrees for monomial maps $\P^n\dashrightarrow\P^n$, and
\cite{LIN_WULCAN}, Lin-Wulcan study the problem of stabilizing certain
monomial maps $\P^n\dashrightarrow \P^n$.  There is also a notion of the
{\em{arithmetic degree}} (of a point) for dominant rational maps
$\P^n\dashrightarrow\P^n$ defined in~\cite{SILVERMAN}. 

\vspace{0.05in}
\noindent{\bf Motivation.} 
The maps $f:X^N\dashrightarrow X^N$ in Theorem \ref{algstabthm} constitute
a new family of examples for which algebraic stability is known and for which
all of the dynamical degrees can be systematically computed (with enough
book-keeping).  They also fit nicely within the context of stabilization, since
Kapranov's Theorem \cite{KAPRANOV} expresses $X^N$ as an iterated blow up
of the projective space $\P^N$.  We initially studied (conjugates of) these
mappings on $\P^N$ and later discovered that {\em all of them} stabilize when lifted to $X^N$.

Moreover, 
the maps $f_\rho:X^N\dashrightarrow X^N$ naturally
arise in the setting of Thurston's topological characterization of rational
maps \cite{sarah}. As a general rule, dynamical quantities associated to
iterating the maps $f_\rho$ should correspond to dynamical quantities
associated to iterating the {\em{Thurston pullback map}} on a Teichm\"uller
space.

\vspace{0.05in}
\noindent
{\bf Outline.}
We begin the paper in Section \ref{action_on_cohomology} with some background 
on the action of a rational map $f: X \dashrightarrow
X$ on cohomology and statement of the criterion for functoriality of a composition
(Proposition \ref{PROP:FUNCTORIAL}) that will be used to proof Theorem~\ref{THM_FUNCTORIAL_SEMIGROUP}.  In Section \ref{SEC:MODULI_SPACE}
we discuss several important properties of the moduli space ${\overline{\cal M}}_{0,n}$,
including Keel's Theorem and Kapranov's Theorem.  Basic properties of the mapping
$f_\rho$ are presented in Section \ref{maps_sect}.  Section \ref{SEC:STABILITY} is dedicated
to the proof of Theorem \ref{algstabthm}.  Computations of $(f_\rho)^\ast$ are done
in Section \ref{SEC:DYN_DEG}.  A catalog of dynamical degrees for specific examples
is presented in Section \ref{SEC:CATALOG}.  

\vspace{0.05in}
\noindent
{\bf Acknowledgments.}
We are very grateful to 
Omar Antol\'in Camarena for convincing us to use the Sage computer algebra program in our calculations and for helping us to write the scripts.  We are also very grateful to Tuyen Truong
who informed us of the universal property for blow ups, which plays a central role in the proof of Theorem \ref{algstabthm}.  We have benefited substantially from discussions with Eric Bedford, Xavier Buff, and Jeffrey Diller, and Kevin Pilgrim.
We also thank the anonymous referee for his or her careful reading of this paper and for making several comments that have improved its exposition.

The research of the first author was supported in part by the NSF DMS-1300315 and the Alfred P. Sloan Foundation.
The research of the second author was supported in part by NSF grant DMS-1102597 and startup funds from the Department
of Mathematics at IUPUI.

%% file: action_on_cohomology.tex
\section{Action on cohomology}\label{action_on_cohomology}
We begin by explaining how a dominant rational map $f:X\dashrightarrow Y$
between smooth complex projective varieties of dimension $N$ induces a well-defined pullback
$f^\ast:H^{k,k}(Y;\C)\to H^{k,k}(X;\C)$ even though $f$ may have a nonempty indeterminacy set $I_f$ (necessarily of codimension 2).  We will first work with the singular cohomology $H^{i}(Y;\C)\to H^i(X;\C)$, and we will then remark about why this definition preserves bidegree.

For the remainder of the paper we will use the term {\em projective manifold} to mean smooth, compact, complex projective variety.
If $V$ is a $k$-dimensional subvariety of a projective manifold
$X$ of dimension $N$, then $V$ determines a {\em fundamental homology class}
$\{V\} \in H_{2k}(X;\C)$.
%, because the singular set of $V$ is of
%real-codimension $\geq 2$.  
The {\em fundamental cohomology class} of $V$ is
$[V] := \PD^{-1}(\{V\}) \in H^{2N-2k}(X;C)$, where $\PD_M: H^j(M) \rightarrow
H_{({\dim}_\mathbb{R}(M)-j)}(M)$ denotes the Poincar\'e duality isomorphism on
a manifold $M$.

Let 
\begin{eqnarray*}
\Gamma_f = \overline{\{(x,y) \in X \times Y \, : \, x \not \in I_f \, \mbox{and} \, y = f(x) \}}
\end{eqnarray*}
 be the {\em graph} of $f$ and let $[\Gamma_f] \in
H^{2N}(X \times Y; \C)$ denote its fundamental cohomology class. Let $\pi_1:X\times Y\to X$ and $\pi_2:X\times Y\to Y$ denote the canonical projection maps.  For any $\alpha \in H^{i}(Y;\C)$, one defines
\begin{eqnarray}\label{EQN:PULLBACK}
f^* \alpha := \pi_{1*}([\Gamma_f] \cupprod \pi_2^* \alpha).
\end{eqnarray}
Here, $\pi_2^*$ is the classical pullback on cohomology, as defined for regular
maps, and $\pi_{1*}$ is the pushforward on cohomology, defined by $\pi_{1*} =
\PD_X^{-1} \circ \pi_{1\#} \circ \PD_{X \times Y}$, where $\pi_{1\#}$ denotes
the push forward on homology.
If $f$ is regular
(i.e. $I_f = \emptyset$) then (\ref{EQN:PULLBACK}) coincides with the classical
definition of pullback.

Suppose that there exits an projective manifold $\tilde{X}$ and holomorphic maps $\pr$ and $\tilde{f}$ making the following
diagram commute (wherever $f \circ \pr$ is defined)
\begin{equation}\label{PULLBACK_DIAGRAM}
\xymatrix{\tilde{X} \ar[d]^\pr  \ar[dr]^{\tilde{f}}  & \\
X \ar @{-->}[r]^f & Y }
\end{equation}
Then, one can show that
\begin{eqnarray}\label{EQN:PULLBACK2}
f^* \alpha = \pr_* \left(\tilde{f}\right)^* \alpha;
\end{eqnarray}
see, for example, \cite[Lemma 3.1]{ROEDER_FUNC}.
Notice that for any rational map $f: X \dashrightarrow Y$, the space $\tilde{X}$ and maps $\pr$ and $\tilde{f}$ always exist: for example, $\tilde{X}$ can be obtained as a desingularization of $\Gamma_f$, with the maps $\pr$ and $\tilde{f}$ corresponding to the lifts of $\pi_1|_{\Gamma_f}$ and $\pi_2|_{\Gamma_f}$.
\[
\xymatrix{
& &\tilde X\ar[d] \ar@/^+2pc/[rdd]^{\tilde{f}} \ar@/^-2pc/[ldd]_{\pr}\\
& &\Gamma_f\ar[rd]^{\pi_2|_{\Gamma_f}} \ar[ld]_{\pi_1|_{\Gamma_f}} \\
&X  \ar @{-->}[rr]^f &  &Y
}
\]

For any K\"ahler manifold $X$, there is a natural isomorphism  $$\bigoplus_{p+q =
i} H^{p,q}(X;\C) \rightarrow H^i(X;\C),$$ where the former are the Dolbeault
cohomology groups and the latter is the singular cohomology.  This induces a
splitting on the singular cohomology of $X$ into bidegrees.  To see that
(\ref{EQN:PULLBACK}) preserves this splitting, observe that
(\ref{EQN:PULLBACK2}) can be applied to any $\overline{\partial}$-closed
$(p,q)$-form $\beta$, with $\left(\tilde{f}\right)^*$ interpreted as the
pullback on smooth forms and $\pr_*$ interpreted as the proper push-forward on
currents of degree $(p,q)$.  As both of these operations induce a well-defined map on cohomology, we see that the $f^* [\beta] = \pr_* \left(\tilde{f}\right)^*
[\beta] = \left[\pr_* \left(\tilde{f}\right)^* \beta\right] \in H^{p,q}(X;\C)$.  In
particular, the pullback defined by (\ref{EQN:PULLBACK}) can be used in the
definition of the dynamical degree $\lambda_k$ for any $1 \leq k \leq N$.

We note that many authors define the pullback on cohomology using forms and
currents as above, rather than the singular cohomology approach we have used.
For more discussion of the latter approach, see \cite{ROEDER_FUNC}.

We will use the following criteria for functoriality of compositions, which is proved in \cite{DS_SUPER,AMERIK,ROEDER_FUNC}:
\begin{proposition}\label{PROP:FUNCTORIAL}
Let $X,Y,$ and $Z$ be projective manifolds of equal dimension, and let $f: X \dashrightarrow Y$ and $g: Y \dashrightarrow Z$ be dominant rational maps.
Suppose that there exits a projective manifold  $\tilde{X}$ and holomorphic maps $\pr$ and $\tilde{f}$ making the following
diagram commute (wherever $f \circ \pr$ is defined)
\begin{equation}
\xymatrix{\tilde{X} \ar[d]^\pr  \ar[dr]^{\tilde{f}}  &  &\\
X \ar @{-->}[r]^f & Y \ar @{-->}[r]^g & Z}
\end{equation}
with the property that $\tilde{f}^{-1}(x)$ is a finite set for every $y \in Y$.  Then, $(g\circ f)^* = f^* \circ g^*$ on all cohomology groups.
\end{proposition}
\begin{remark}
Note that it follows from the criterion of Bedford-Kim \cite[Thm.
1.1]{BEDFORD_KIM2} that if $f_{X \setminus I_f} : X \setminus I_f \rightarrow
X$ is finite then $f$ is $1$-stable.  This is not sufficient for $k$-stability,
when $k > 1$, as shown in \cite[Prop. 6.1]{ROEDER_FUNC}.
This is why we use the stronger sufficient condition in Proposition \ref{PROP:FUNCTORIAL}.
\end{remark}

The following lemma (see, e.g. \cite[Lem. 19.1.2]{FULTON})  will be helpful when using (\ref{EQN:PULLBACK2}) to compute pullbacks.

\begin{lemma}\label{LEM:PUSH_FORWARD_VARIETIES}
Suppose that $f: X \rightarrow Y$ is a proper holomorphic map between projective manifolds.  For any irreducible subvariety $V \subseteq X$
we have
\begin{itemize}
\item[(i)] if $\dim(f(V)) = \dim(V)$, then $f_*([V]) = {\rm deg}_{\rm top}(f|_{V}) [f(V)]$,
where ${\rm deg}_{\rm top}(f|_{V})$ is the number of preimages under $f|_V$ of a generic point from $f(V)$.
\item[(ii)] Otherwise, $f_*([V]) = 0$.
\end{itemize}
\end{lemma}

%\begin{proof}
%This is essentially \cite[Lem. 19.1.2]{FULTON} combined the remark in Section
%1.4 of \cite{FULTON} that $\deg(V/f(V))$ is equal to the topological degree
%${\rm deg}_{\rm top}(f|_{V})$ of $f|_V: V \rightarrow f(V)$.
%\end{proof}

%% file: DM.tex
\section{Moduli space}\label{SEC:MODULI_SPACE}

Let $P=\{p_1,\ldots,p_n\}$ be a finite set consisting of at least three points. The {\em{moduli space of genus $0$ curves marked by $P$}} is by definition
\[
\mathcal M_P:=\{\varphi:P\hookrightarrow\P^1\text{ up to postcomposition by M\"obius transformations}\}.
\]

\subsection{Projective space}\label{projective}
Every element of $\mathcal M_P$ has a representative $\varphi:P\hookrightarrow \P^1$ so that 
\[
\varphi(p_1)=0\text{ and }\varphi(p_2)=\infty,
\]
and the point $[\varphi]\in \mathcal M_P$ is determined by the $(n-2)$-tuple 
\[
\left(z_1,\ldots,z_{n-2}\right)\in\C^{n-2}\quad\text{where $z_i:=\varphi(p_{i+2})\quad \text{for }1\leq i\leq n-2$},
\]
up to scaling by a nonzero complex number. In other words, the point $[\varphi]\in\mathcal M_P$ is uniquely determined by $[z_1:\cdots:z_{N+1}]\in\P^N$, where $N:=n-3$. There are some immediate constraints on the complex numbers $z_i$ in order to ensure that $\varphi:P\hookrightarrow\P^1$ is injective. Indeed, $\mathcal M_P$ is isomorphic to the complement of $(n-1)(n-2)/2$ hyperplanes in $\P^N$. We state this in the following proposition. 
\begin{proposition}\label{modspaceiso}
Define $z_0:=0$. The moduli space $\mathcal M_P$ is isomorphic to $\P^N \setminus \Delta$ where $\Delta$ is the following collection of hyperplanes
\[
\Delta:=\left\{z_i=z_j\;|\;0\leq i<j\leq N+1\right\}.
\]
In particular, $\mathcal M_P$ is a complex manifold of dimension $N$. 
\end{proposition}
\begin{proof}
This follows immediately from the normalization above. 
\end{proof}
\medbreak
The following fact is straight-forward, but we state it explicitly as is will be used in subsequent sections. 
\begin{proposition}\label{PQprop}
Let $Q=\{q_1,\ldots,q_n\}$, and let $\iota:P\to Q$ be a bijection. Then $\iota$ induces an isomorphism 
\[
\iota^*:\mathcal M_Q\to \mathcal M_P
\]
\end{proposition}
\begin{proof}
Let $m\in \mathcal M_Q$, and let $\varphi:Q\hookrightarrow\P^1$ be a representative of $m$. Then $\iota^*(m)\in\mathcal M_P$ is represented by $\varphi \circ \iota:P\hookrightarrow\P^1$. 
\end{proof}
\medbreak
The moduli space $\mathcal M_P$ is not compact.  

\subsection{The Deligne-Mumford compactification}\label{DMCsubsect}
A {\em{stable curve of genus $0$ marked by $P$}} is an injection $\varphi:P\hookrightarrow C$ where $C$ is a connected algebraic curve whose singularities are ordinary double points (called {\em{nodes}}), such that 
\begin{enumerate}
\item each irreducible component is isomorphic to $\P^1$, 
\item\label{graph} the graph, $G_C$, whose vertices are the irreducible components and whose edges connect components intersecting at a node, is a tree
\item for all $p\in P$, $\varphi(p)$ is a smooth point of $C$, and 
\item the number of marked points plus nodes on each irreducible component of $C$ is at least three. 
\end{enumerate}
The marked stable curves $\varphi_1:P\hookrightarrow C_1$ and
$\varphi_2:P\hookrightarrow C_2$ are isomorphic if there is an isomorphism
$\mu:C_1\to C_2$ such that $\varphi_2=\mu\circ \varphi_1$. The set of stable
curves of genus $0$ marked by $P$ modulo isomorphism can be given the structure
of a smooth projective variety \cite{knudsen,knudsenmumford}, called the
{\em{Deligne-Mumford compactification}}, and denoted by $\overline{\mathcal
M}_P$. The moduli space $\mathcal M_P$ is an open Zariski dense subset of
$\overline{\mathcal M}_P$.  In this subsection we will state some of the
well-known properties of $\overline{\mathcal M}_P$.

%\subsubsection{Trees and boundary strata}\label{trees_subsection}
%\begin{remark}\label{stratified}

Let $P=\{p_1,\ldots,p_n\}$ be a set with at least three elements. The
compactification divisor of $\mathcal M_P$ in $\overline{\mathcal M}_P$ is the
set of all (isomorphism classes) of marked stable curves with {at least} one
node.  Generic points of  $\overline{\mathcal M}_P \setminus \mathcal M_P$
consist of the  (isomorphism classes) of marked stable curves
$\varphi:P\hookrightarrow C$ with at {exactly} one node.  For each such
generic boundary point, taking $\varphi^{-1}$ of the connected components of $C
\setminus \{\text{node}\}$ induces a partition of $P$ into two sets $S \cup
S^c$, where $S^c:=P\setminus S$. The set of generic boundary components inducing a given partition of $P$
is an irreducible quasiprojective variety and its closure in
$\overline{\mathcal M}_P$ is an algebraic hypersurface denoted $\D^S \equiv \D^{S^c}$, and it is a {boundary divisor} of $\overline{\mathcal M}_P$.

If $|S| = n_1$ and $|S^c| = n_2$ then there is an isomorphism
\[
\D^S \approx \overline{\mathcal M}_{0,n_1+1}\times \overline{\mathcal M}_{0,n_2+1}.
\]
The $n_1+1$ points in the first factor consist of the $n_1$ points of $S$ together with the node, and similarly for the second factor.

Any $k$ distinct boundary divisors $\D^{S_1},\ldots, \D^{S_k}$ intersect transversally and if this intersection is nonempty, the result is an irreducible codimension $k$ boundary stratum.  This corresponds to the set of marked stable curves which induces a stable partition of $P$ into $k+1$ blocks.  There is an analogous description in terms of trees. 

\noindent {\bf Marked Stable Trees.} Let $\varphi:P\hookrightarrow C$ be a marked stable curve of genus $0$. As mentioned in point (\ref{graph}) above, there is a graph $G_C$ associated to $\varphi:P\hookrightarrow C$, which is a tree. Let $V_C$ be the set of vertices of $G_C$; the marking $\varphi:P\hookrightarrow C$ induces a map $\varphi_*:P\to V_C$, sending $p\in P$ to the vertex corresponding to the irreducible component which contains $\varphi(p)$. We will call
\[
T_{\varphi:P\hookrightarrow C}:=(G_C, \, \, \varphi_*:P\to V_C)
\]
the {\em marked stable tree} associated to the marked stable curve $\varphi:P\hookrightarrow C$. More generally, any graph $G$ which is a tree, together with a map $\psi:P\to V_G$ will be a marked stable tree if for all $v\in V_G$, 
\[
\mathrm{degree}(v)+|\psi^{-1}(v)|\geq 3.
\]
Given two generic points $\varphi:P\hookrightarrow C$ and
$\varphi':P\hookrightarrow C'$ of the (nonempty) intersection $\D^{S_1}\cap
\cdots\cap \D^{S_k}$, the trees $T_{\varphi:P\hookrightarrow C}$ and
$T_{\varphi':P\hookrightarrow C'}$ are {\em isomorphic} in the following sense:
there is a graph isomorphism $\beta:G_C\to G_{C'}$ so that
$\varphi'_*=\beta\circ \varphi_*$. The stratum $\D^{S_1}\cap \cdots\cap \D^{S_k}$
can be labeled by the isomorphism class of $T_{\varphi:P\hookrightarrow C}$. 

It is well-known that there is a bijection between the following sets:
\[
\{\text{codimension $k$ boundary strata in $\overline{\mathcal M}_P$}\}\quad\longleftrightarrow
\]
\[
\{\text{isomorphism classes of marked stable trees with $k+1$ vertices}\}
\]
\begin{lemma}\label{intersectingdivslemma}
Let $Z$ be a boundary stratum of codimension $k$ in $\overline{\mathcal M}_P$. There is a unique set $\{\D^{S_1}, \ldots, \D^{S_k}\}$ of boundary divisors so that $Z=\D^{S_1}\cap \cdots\cap \D^{S_k}$. 
\end{lemma} 
\begin{proof}
This result follows immediately from the remarks above. 
\end{proof}
\medbreak

%Let $\varphi:P\hookrightarrow C$ be a marked stable curve corresponding to a generic point of $Z$, and consider the tree $T_{\varphi:P\hookrightarrow C}$. To each edge $e$ of the graph $G_C$, we associate a boundary divisor $\D_e$ in the following way: remove $e$ from $G_C$, obtaining two connected components: $E_1$ and $E_2$. The following two sets 
%\[
%(\varphi_*)^{-1}(V_C\cap E_1)\quad\text{and}\quad(\varphi_*)^{-1}(V_C\cap E_2)
%\]
%partition $P$ into two blocks; set $S:=(\varphi_*)^{-1}(V_C\cap E_1)$, set $\D_e:=\D^S$. 

% \begin{remark}\label{modspacename}
%Proposition \ref{PQprop} implies that if $|P|=|Q|$, then $\mathcal M_P$ and $\mathcal M_Q$ are isomorphic. It follows that $\overline{\mathcal M}_P$ and $\overline{\mathcal M}_Q$ are also isomorphic. Hence the notation $\mathcal M_{0,n}$ (and respectively $\overline{\mathcal M}_{0,n}$) is often used, where $n=|P|=|Q|$. We will use the notation $\mathcal M_P$ (respectively $\overline{\mathcal M}_P$) when the set $P$ is explicit, and we will use the notation $\mathcal M_{0,n}$ (respectively $\overline{\mathcal M}_{0,n}$) otherwise. 
%\end{remark}

\subsection{Keel's theorem}

In \cite{KEEL}, Keel exhibits generators and relations for the cohomology ring
of $\overline{\mathcal M}_P$. Let $[\D^S]$ denote the fundamental cohomology class of the boundary divisor
$\D^S$.

\begin{theorem}[Keel, \cite{KEEL}]\label{keelthm}
The cohomology ring $H^*(\overline{\mathcal M}_P;\C)$ is the ring 
\[
\Z\left[[\D^S]\;:\; S\subseteq P, \; \, |S|,|S^c|\geq 2\right]
\]
modulo the following relations:
\begin{enumerate}
\item $[\D^S]=[\D^{S^c}]$
\item For any four distinct $p_i,p_j,p_k,p_l\in P$, we have 
\[
\sum_{\stackrel{p_i,p_j\in S}{p_k,p_l\in S^c}} [\D^S]=\sum_{\stackrel{p_i,p_k\in S}{p_j,p_l\in S^c}}[\D^S] =\sum_{\stackrel{p_i,p_l\in S}{p_j,p_k\in S^c}} [\D^S].
\]
\item $[\D^S]\cupprod [\D^T]=0$ unless one of the following holds:
\[
S\subseteq T,\quad T\subseteq S, \quad S\subseteq T^c,\quad T^c\subseteq S.
\]
\end{enumerate}
\end{theorem}

Implicit in Keel's Theorem is the assertion that the codimension $k$ boundary strata are complete intersections:

\begin{corollary}\label{completeintersection} We have
\begin{eqnarray*}
[\D^{S_1} \cap \cdots \cap \D^{S_k}] = [\D^{S_1}] \cupprod \cdots \cupprod [\D^{S_k}].
\end{eqnarray*}
\end{corollary}

We now construct Kapranov's space $X^N$ which is isomorphic to $\overline{\mathcal M}_P$. 

\subsection{Kapranov's Theorem}\label{kapsect}
We may choose coordinates and identify $\mathcal M_P$ with $\P^N \setminus \Delta$ as
stated in Proposition \ref{modspaceiso}. In this concrete setting, there is a
description of $\overline{\mathcal{M}}_P$ as a {\em{sequential blow up}} of
$\P^N$ due to Kapranov \cite{KAPRANOV}.
% There is also related work by Harvey and Lloyd-Philipps \cite{Harvey}. 

Normalize to identify $\mathcal M_P$ with $\P^N \setminus \Delta$ as in Proposition \ref{modspaceiso}, and consider the following subsets of $\P^N$. Let 
\[
A^0:=\left\{[1:0:\cdots:0],[0:1:0:\cdots:0],\ldots,[0:\cdots:0:1],[1:1:\cdots:1]\right\},
\]
and for $1\leq i\leq N-2$, let $A^i$ be the set of all $\binom{N+2}{i+1}$ projective linear subspaces of dimension $i$ in $\P^N$, which are spanned by collections of $i+1$ of distinct points in $A^0$. Let $X_0:=\P^N,$ and for each $0\leq i\leq N-2$ define $\alpha^i:X_{i+1}\to X_i$ to be the blow up of of $X_i$ along the proper transform $\widetilde A^i$ of $A^i$ under $\alpha^{0}\circ \cdots \circ \alpha^{i-1}$. 
\begin{theorem}[Kapranov, \cite{KAPRANOV}]\label{kapranovthm}
Let $P=\{p_1,\ldots,p_n\}$ contain at least three points. Normalize to identify $\mathcal M_P$ with $\P^N \setminus \Delta$ where $N=n-3$ as in Proposition \ref{modspaceiso}. Then the Deligne-Mumford compactification $\overline{\mathcal M}_P$ is isomorphic to the space $X^N:=X_{N-1}$ constructed above. 
\end{theorem}

The proof of Kapranov's Theorem is a bit subtle, establishing the isomorphism
using the space of Veronese curves.  For a different perspective, closer to that of the present paper, we refer the reader to the paper of Harvey
and Lloyd-Philipps \cite{Harvey}.

\begin{remark}\label{divnotation}
Via the isomorphism $\overline{\mathcal M}_P\approx X^N$ from Theorem \ref{kapranovthm}, we will use Theorem \ref{keelthm}  to find appropriate
bases for the cohomology groups $H^{k,k}(X^N;\C)$ in Section \ref{SEC:DYN_DEG}. To this end, we adopt the following notation. Let $\D^S\subseteq \overline{\mathcal M}_P$ be a boundary divisor. We will use the notation $D^S\subseteq X^N$ to denote the image of $\D^S$ under the explicit isomorphism $\overline{\mathcal M}_P\approx X^N$ from Theorem \ref{kapranovthm}.
\end{remark}

For $|P|=3$, $\mathcal M_P=\overline{\mathcal M}_P$ is a point. For $|P|=4$, $\mathcal M_P$ is isomorphic to $\P^1 \setminus \{0,1,\infty\}$, and $\overline{\mathcal M}_P$ is isomorphic to $\P^1$. 

\begin{example}\label{5points}
Let $P=\{p_1,p_2,p_3,p_4,p_5\}$. Following Proposition \ref{modspaceiso}, $\mathcal M_P$ is isomorphic to $\P^2 \setminus \Delta$, where 
\[
\Delta=\{z_1=0,z_2=0,z_3=0,z_1=z_2,z_2=z_3,z_1=z_3\}.
\]
The space $\overline{\mathcal M}_P$ is isomorphic to $X^2$, which is equal to $\P^2$ blown up at the four points comprising $A^0$:
\[
\{[1:0:0],[0:1:0],[0:0:1],[1:1:1]\}.
\]
There are $10$ boundary divisors in $X^2$: the proper transforms of the six lines comprising $\Delta$, plus the four exceptional divisors. The ten boundary divisors correspond to the $\binom{5}{2}$ stable partitions of $P$ into two blocks.  The space $X^2$ is depicted in Figure~\ref{FIG_X2}.

\begin{figure}[h]
\scalebox{0.8}{
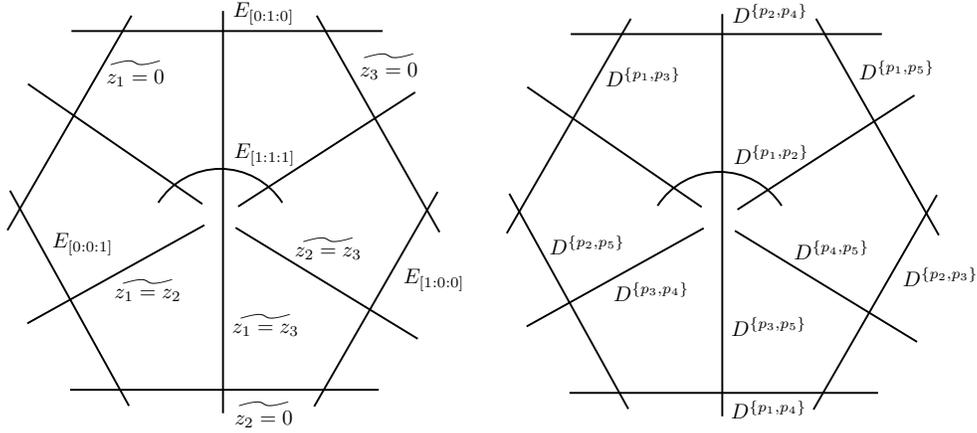
}
\caption{\label{FIG_X2} Depiction of $X^2$.  Left: boundary divisors are labeled as proper transforms of lines in $\P^2$ and exceptional divisors.  Right: boundary divisors are labeled according to Remark \ref{divnotation}.}
\end{figure}
\end{example}

\begin{example}
If $|P| = 6$, then $\overline{\mathcal M}_P$ is isomorphic to $X^3$, the sequential blow up of $\P^3$ where
\[
A^0=\{[0:0:0:1],[0:0:1:0],[0:1:0:0],[1:0:0:0],[1:1:1:1]\},
\]
$A^1$ is the set of $10 = \binom{5}{2}$ lines spanned by pairs of points in $A^0$.  A depiction of $X^3$ is shown in Figure~\ref{FIG_X3}

\begin{figure}[h]
\begin{center}
\scalebox{0.7}{
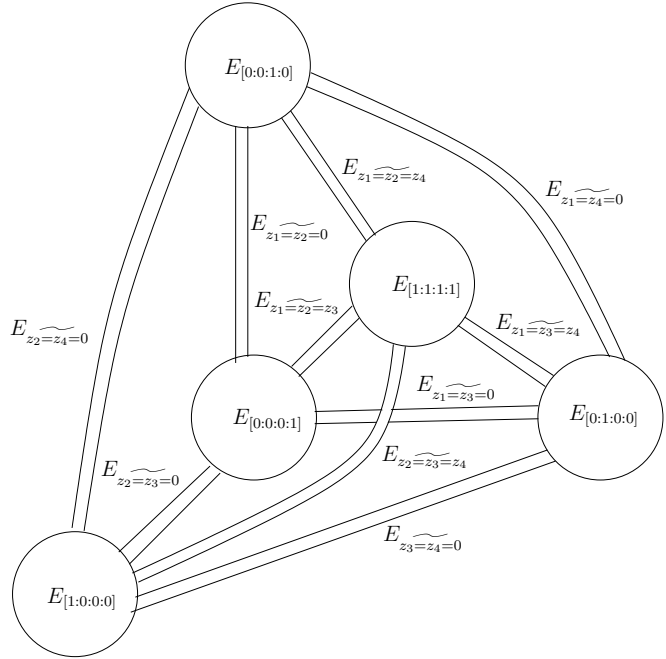
}
\end{center}
\caption{\label{FIG_X3} Depiction of $X^3$ with all boundary divisors corresponding
to exceptional divisors over $A^0$ and over proper transforms of lines from $A^1$ labeled.  (The remaining
$10$ boundary divisors corresponding to hyperplanes in $\P^3$ are not labeled.)}
\end{figure}
\end{example}

\begin{example}
If $|P|=7$, then $\overline{\mathcal M}_P$ is isomorphic to $X^4$, the sequential blow up of $\P^4$ where 
{\footnotesize
\begin{eqnarray*}
A^0=\{[0:0:0:0:1],[0:0:0:1:0],[0:0:1:0:0],[0:1:0:0:0],[1:0:0:0:0],[1:1:1:1:1]\},
\end{eqnarray*}
}
$A^1$ is the set of $15$ lines spanned by pairs of points in $A^0$, and $A^2$ is the set of $20$ planes spanned by triples of points in $A^0$. Let $[x_1:x_2:x_3:x_4:x_5]\in\P^4$, and let $L$ be the element of $A^2$ spanned by 
\[
\{[0:0:0:0:1],[0:0:0:1:0],[0:0:1:0:0]\},% \text{ that is, the locus in $\P^4$  given by }x_1=x_2=0,
\]
 that is, the locus in $\P^4$  given by $x_1=x_2=0$,
and let $M$ be the element in $A^2$ spanned by 
\[
\{[0:1:0:0:0],[1:0:0:0:0],[1:1:1:1:1]\},% \text{ that is, the locus in $\P^4$ given by }x_3=x_4=x_5.
\]
that is, the locus in $\P^4$ given by $x_3=x_4=x_5$.
Note that $L$ and $M$ intersect at the point $[0:0:1:1:1]$ which is not in $A^0 \cup A^1$, so {\em{a priori}}, the order of the blow ups in the construction above might matter in constructing the space $X^4$. However, this is not the case since $L$ and $M$ intersect transversally. Indeed, this phenomenon occurs in the general setting for $|P|$ arbitrary, but these intersections are always transverse and are therefore irrelevant in the blow up construction (see Lemma \ref{LEM:COMMUTING_BLOW_UPS} and Lemma \ref{LEM:GOODCENTERS1}). 
\end{example}

\subsection{Automorphisms of $X^N$} 
The automorphism group of $X^N$ is clearly isomorphic to the automorphism group of $\overline{\mathcal M}_P$. We will study automorphisms of $\overline{\mathcal M}_P$ that extend automorphisms of $\mathcal M_P$. If $|P|=4$, then the automorphisms of $\overline{\mathcal M}_P\approx \P^1$ that extend the automorphisms of $\mathcal M_P$ consist of the M\"obius transformations that map the set of three points comprising the boundary of $\mathcal M_P$ in $\overline{\mathcal M}_P$ to itself; that is, $\mathrm{Aut}({\mathcal M}_P)$ is isomorphic to the permutation group on three letters. If $|P|>4$, then $\mathrm{Aut}({\mathcal M}_P)$ is isomorphic to $S_P$, the group of permutations of elements of the set $P$ (see  \cite{xavierreference}, and compare with Proposition \ref{PQprop}). 
\begin{proposition}\label{autos}
Suppose that $|P|>4$, and let $\rho\in S_P$. Then the automorphism $g_\rho:\mathcal M_P\to \mathcal M_P$ extends to an automorphism $g_\rho:\overline{\mathcal M}_P\to \overline{\mathcal M}_P$. 
\end{proposition}
\begin{proof}
Let $P=\{p_1,\ldots,p_n\}$, and let $\rho\in S_P$. The permutation $\rho$ relabels the points in $P$, which effectively just changes coordinates on $ \overline{\mathcal M}_P$. This is evident using Kapranov's theorem from Section \ref{kapsect}. Indeed, in the construction of $X^N$, we began with a choice of normalization: we identified $\mathcal{M}_P$ with a $\P^N \setminus \Delta$ by choosing a representative $\varphi:P\hookrightarrow\P^1$ so that $\varphi(p_1)=0$, and $\varphi(p_2)=\infty$, and setting $z_i:=\varphi(p_{i+2})$ for $1\leq i\leq n-2$, we identified the point $[\varphi]\in\mathcal M_P$ with the point $[z_1:\cdots:z_{N+1}]\in\P^N$. To build $X^N,$ we performed the appropriate sequential blow up of this copy of $\P^N$. 

Carrying out the same construction, but taking the permutation into account, we normalize so that for the representative $\varphi:P\hookrightarrow\P^1$ 
\[
\varphi(p_{\rho^{-1}(1)})=0,\quad\text{and}\quad\varphi(p_{\rho^{-1}(2)})=\infty,
\]
and by setting $z_i:=\varphi(p_{\rho^{-1}(i)})$, for $1\leq i\leq n-2$, we identify the point $[\varphi]\in\mathcal M_P$ with the point $[z_1:\cdots:z_{M+1}]\in\P^M$, where $M:=n-3$. Build a space $Y^M$ which is the sequential blow up of $\P^M$ as prescribed in Section \ref{kapsect} (we have changed notation so as not to confuse the two constructions of the `same' space). The spaces $X^N$ and $Y^M$ are clearly isomorphic, and we see that $\rho$ induces an automorphism $g_\rho:\overline{\mathcal M}_P\to \overline{\mathcal M}_P$ which extends $g_\rho:\mathcal M_P \to \mathcal M_P$. 
\end{proof}
\medbreak

%% file: M05_v2.pstex_t
\begin{picture}(0,0)%
\epsfig{file=M05_v2.pdf}%
\end{picture}%
\setlength{\unitlength}{3947sp}%
\begingroup\makeatletter\ifx\SetFigFont\undefined%
\gdef\SetFigFont#1#2#3#4#5{%
  \reset@font\fontsize{#1}{#2pt}%
  \fontfamily{#3}\fontseries{#4}\fontshape{#5}%
  \selectfont}%
\fi\endgroup%
\begin{picture}(7356,3382)(3253,-3805)
\put(3627,-2368){\makebox(0,0)[lb]{\smash{{\SetFigFont{10}{12.0}{\familydefault}{\mddefault}{\updefault}{\color[rgb]{0,0,0}$E_{[0:0:1]}$}%
}}}}
\put(5531,-2437){\makebox(0,0)[lb]{\smash{{\SetFigFont{10}{12.0}{\familydefault}{\mddefault}{\updefault}{\color[rgb]{0,0,0}$\widetilde{z_2=z_3}$}%
}}}}
\put(5039,-3037){\makebox(0,0)[lb]{\smash{{\SetFigFont{10}{12.0}{\familydefault}{\mddefault}{\updefault}{\color[rgb]{0,0,0}$\widetilde{z_1=z_3}$}%
}}}}
\put(6385,-2646){\makebox(0,0)[lb]{\smash{{\SetFigFont{10}{12.0}{\familydefault}{\mddefault}{\updefault}{\color[rgb]{0,0,0}$E_{[1:0:0]}$}%
}}}}
\put(4051,-1076){\makebox(0,0)[lb]{\smash{{\SetFigFont{10}{12.0}{\familydefault}{\mddefault}{\updefault}{\color[rgb]{0,0,0}$\widetilde{z_1=0}$}%
}}}}
\put(4115,-2761){\makebox(0,0)[lb]{\smash{{\SetFigFont{10}{12.0}{\familydefault}{\mddefault}{\updefault}{\color[rgb]{0,0,0}$\widetilde{z_1=z_2}$}%
}}}}
\put(6042,-1021){\makebox(0,0)[lb]{\smash{{\SetFigFont{10}{12.0}{\familydefault}{\mddefault}{\updefault}{\color[rgb]{0,0,0}$\widetilde{z_3=0}$}%
}}}}
\put(5059,-3756){\makebox(0,0)[lb]{\smash{{\SetFigFont{10}{12.0}{\familydefault}{\mddefault}{\updefault}{\color[rgb]{0,0,0}$\widetilde{z_2=0}$}%
}}}}
\put(5047,-555){\makebox(0,0)[lb]{\smash{{\SetFigFont{10}{12.0}{\familydefault}{\mddefault}{\updefault}{\color[rgb]{0,0,0}$E_{[0:1:0]}$}%
}}}}
\put(5056,-1665){\makebox(0,0)[lb]{\smash{{\SetFigFont{10}{12.0}{\familydefault}{\mddefault}{\updefault}{\color[rgb]{0,0,0}$E_{[1:1:1]}$}%
}}}}
\put(9451,-2461){\makebox(0,0)[lb]{\smash{{\SetFigFont{10}{12.0}{\familydefault}{\mddefault}{\updefault}{\color[rgb]{0,0,0}$D^{\{p_4,p_5\}}$}%
}}}}
\put(9962,-1045){\makebox(0,0)[lb]{\smash{{\SetFigFont{10}{12.0}{\familydefault}{\mddefault}{\updefault}{\color[rgb]{0,0,0}$D^{\{p_1,p_5\}}$}%
}}}}
\put(8976,-1711){\makebox(0,0)[lb]{\smash{{\SetFigFont{10}{12.0}{\familydefault}{\mddefault}{\updefault}{\color[rgb]{0,0,0}$D^{\{p_1,p_2\}}$}%
}}}}
\put(8959,-3061){\makebox(0,0)[lb]{\smash{{\SetFigFont{10}{12.0}{\familydefault}{\mddefault}{\updefault}{\color[rgb]{0,0,0}$D^{\{p_3,p_5\}}$}%
}}}}
\put(10305,-2670){\makebox(0,0)[lb]{\smash{{\SetFigFont{10}{12.0}{\familydefault}{\mddefault}{\updefault}{\color[rgb]{0,0,0}$D^{\{p_2,p_3\}}$}%
}}}}
\put(7971,-1100){\makebox(0,0)[lb]{\smash{{\SetFigFont{10}{12.0}{\familydefault}{\mddefault}{\updefault}{\color[rgb]{0,0,0}$D^{\{p_1,p_3\}}$}%
}}}}
\put(8957,-3712){\makebox(0,0)[lb]{\smash{{\SetFigFont{10}{12.0}{\familydefault}{\mddefault}{\updefault}{\color[rgb]{0,0,0}$D^{\{p_1,p_4\}}$}%
}}}}
\put(8967,-602){\makebox(0,0)[lb]{\smash{{\SetFigFont{10}{12.0}{\familydefault}{\mddefault}{\updefault}{\color[rgb]{0,0,0}$D^{\{p_2,p_4\}}$}%
}}}}
\put(8035,-2785){\makebox(0,0)[lb]{\smash{{\SetFigFont{10}{12.0}{\familydefault}{\mddefault}{\updefault}{\color[rgb]{0,0,0}$D^{\{p_3,p_4\}}$}%
}}}}
\put(7532,-2419){\makebox(0,0)[lb]{\smash{{\SetFigFont{10}{12.0}{\familydefault}{\mddefault}{\updefault}{\color[rgb]{0,0,0}$D^{\{p_2,p_5\}}$}%
}}}}
\end{picture}%

%% file: M06.pstex_t
\begin{picture}(0,0)%
\epsfig{file=M06.pdf}%
\end{picture}%
\setlength{\unitlength}{3947sp}%
\begingroup\makeatletter\ifx\SetFigFont\undefined%
\gdef\SetFigFont#1#2#3#4#5{%
  \reset@font\fontsize{#1}{#2pt}%
  \fontfamily{#3}\fontseries{#4}\fontshape{#5}%
  \selectfont}%
\fi\endgroup%
\begin{picture}(6230,5887)(2053,-7404)
\put(5697,-5011){\makebox(0,0)[lb]{\smash{{\SetFigFont{12}{14.4}{\familydefault}{\mddefault}{\updefault}{\color[rgb]{0,0,0}$E_{\widetilde{z_1=z_3=0}}$}%
}}}}
\put(5387,-5635){\makebox(0,0)[lb]{\smash{{\SetFigFont{12}{14.4}{\familydefault}{\mddefault}{\updefault}{\color[rgb]{0,0,0}$E_{\widetilde{z_2=z_3=z_4}}$}%
}}}}
\put(6406,-4450){\makebox(0,0)[lb]{\smash{{\SetFigFont{12}{14.4}{\familydefault}{\mddefault}{\updefault}{\color[rgb]{0,0,0}$E_{\widetilde{z_1=z_3=z_4}}$}%
}}}}
\put(2385,-6912){\makebox(0,0)[lb]{\smash{{\SetFigFont{12}{14.4}{\familydefault}{\mddefault}{\updefault}{\color[rgb]{0,0,0}$E_{[1:0:0:0]}$}%
}}}}
\put(6852,-3259){\makebox(0,0)[lb]{\smash{{\SetFigFont{12}{14.4}{\familydefault}{\mddefault}{\updefault}{\color[rgb]{0,0,0}$E_{\widetilde{z_1=z_4=0}}$}%
}}}}
\put(2053,-4506){\makebox(0,0)[lb]{\smash{{\SetFigFont{12}{14.4}{\familydefault}{\mddefault}{\updefault}{\color[rgb]{0,0,0}$E_{\widetilde{z_2=z_4=0}}$}%
}}}}
\put(2869,-5776){\makebox(0,0)[lb]{\smash{{\SetFigFont{12}{14.4}{\familydefault}{\mddefault}{\updefault}{\color[rgb]{0,0,0}$E_{\widetilde{z_2=z_3=0}}$}%
}}}}
\put(7068,-5270){\makebox(0,0)[lb]{\smash{{\SetFigFont{12}{14.4}{\familydefault}{\mddefault}{\updefault}{\color[rgb]{0,0,0}$E_{[0:1:0:0]}$}%
}}}}
\put(5026,-3061){\makebox(0,0)[lb]{\smash{{\SetFigFont{12}{14.4}{\familydefault}{\mddefault}{\updefault}{\color[rgb]{0,0,0}$E_{\widetilde{z_1=z_2=z_4}}$}%
}}}}
\put(5401,-6361){\makebox(0,0)[lb]{\smash{{\SetFigFont{12}{14.4}{\familydefault}{\mddefault}{\updefault}{\color[rgb]{0,0,0}$E_{\widetilde{z_3=z_4=0}}$}%
}}}}
\put(3976,-2161){\makebox(0,0)[lb]{\smash{{\SetFigFont{12}{14.4}{\familydefault}{\mddefault}{\updefault}{\color[rgb]{0,0,0}$E_{[0:0:1:0]}$}%
}}}}
\put(5476,-4111){\makebox(0,0)[lb]{\smash{{\SetFigFont{12}{14.4}{\familydefault}{\mddefault}{\updefault}{\color[rgb]{0,0,0}$E_{[1:1:1:1]}$}%
}}}}
\put(4051,-5311){\makebox(0,0)[lb]{\smash{{\SetFigFont{12}{14.4}{\familydefault}{\mddefault}{\updefault}{\color[rgb]{0,0,0}$E_{[0:0:0:1]}$}%
}}}}
\put(4240,-4252){\makebox(0,0)[lb]{\smash{{\SetFigFont{12}{14.4}{\familydefault}{\mddefault}{\updefault}{\color[rgb]{0,0,0}$E_{\widetilde{z_1=z_2=z_3}}$}%
}}}}
\put(4204,-3568){\makebox(0,0)[lb]{\smash{{\SetFigFont{12}{14.4}{\familydefault}{\mddefault}{\updefault}{\color[rgb]{0,0,0}$E_{\widetilde{z_1=z_2=0}}$}%
}}}}
\end{picture}%

%% file: maps.tex
\section{The maps $f_\rho:X^N\dashrightarrow X^N$}\label{maps_sect}

As previously mentioned, the maps $f_\rho:X^N\dashrightarrow X^N$ will be a composition of two maps: an automorphism $g_\rho:X^N\to X^N$ and a map $s:X^N\dashrightarrow X^N$, which we now define. 

Let $P=\{p_1,p_2,\ldots,p_n\}$, and normalize to identify $\mathcal M_P$ with $\P^N \setminus \Delta$ as in Proposition \ref{modspaceiso}, and via Kapranov's construction (Theorem \ref{kapranovthm}), build the space $X^N$ as a sequential blow up of $\P^N$. Consider the squaring map $s_0:\P^N\to \P^N$ given by 
\[
s_0:[z_1:\cdots:z_{N+1}]\mapsto [z_1^2:\cdots z_{N+1}^2],
\]
which is clearly holomorphic. Note that the critical locus of $s_0$ consists precisely of the union of hyperplanes 
\[
{\rm Crit}(s_0) = \bigcup_{i=1}^{N+1} \{z_i=0\}.
\]
Moreover, every component of $\Delta$ is mapped to itself by $s_0$. 

The map $s:X^N\dashrightarrow X^N$ is simply the lift of $s_0:\P^N\to\P^N$
under the map $\mathcal A:=\alpha_0\circ\cdots\circ\alpha_{N-2}:X_{N-1}\to X_0$
where $X^N:=X_{N-1}$ and $X_0:=\P^N$ in the Kapranov construction (see Theorem
\ref{kapranovthm}). The map $s:X^N\dashrightarrow X^N$ is not holomorphic
(unless $N=1$, or equivalently, $|P|=4$); indeed, there are points of
indeterminacy arising from extra preimages of varieties that were previously
blown up. For example, consider $P=\{p_1,\ldots,p_5\}$ (as in Example
\ref{5points}). The space $X^2$ is $\P^2$ blown up at
$[0:0:1],[0:1:0],[1:0:0]$, and $[1:1:1]$. Let $\zeta\in
(s_0^{-1}([1:1:1]) \setminus \{[1:1:1]\})$. Then $s:X^2\dashrightarrow X^2$ has a point
of indeterminacy at $\alpha_0^{-1}(\zeta)$. 
 In fact, for any $N > 1$ the indeterminacy set for $s: X^N \dashrightarrow X^N$
has dimension $N-2$.  Notice that $\{z_1 = z_2 = -z_3\} \in s_0^{-1}(\{z_1 =
z_2 = z_3\})$, with $\{z_1 = z_2 = z_3\} \in A^{N-2}$.  Since $\{z_1 = z_2 =
-z_3\}$ is not a center of blow up, it's proper transform under $\mathcal{A}$ is in the indeterminacy locus $I_s$.

%
%{\color{red}
%\begin{lemma}\label{CRITLOCUS} The critical locus of $s: X^N \dashrightarrow X^N$ consists of the points
%\begin{eqnarray*}
%{\rm Crit}(s) = {\mathcal A}^{-1}({\rm Crit}(s_0)) = \{D^S \, : \, |S \cap \{p_1,p_2\}| = 1\}
%\end{eqnarray*}
%\end{lemma}
%}
%
%\begin{proof}
%Since $s$ is the lift of $s_0$, the following diagram commutes:
%\begin{eqnarray*}
%\xymatrix{
%X^N \ar @{-->}[r]^{s} \ar[d]^{\mathcal A} & X^N \ar[d]^{\mathcal A} \\
%\P^N \ar[r]^{s_0} & \P^N,
%}
%\end{eqnarray*}
%with $\mathcal A$ unramified.  This establishes the first equality.  Using the coordinates from Section \ref{kapsect} one observes that $\mathcal A^{-1}({\rm Crit}(s_0))$ consists precisely
%of divisors $D^S$ where $|S \cap \{p_1,p_2\}| = 1$.
%\end{proof}

By Proposition \ref{autos}, every permutation $\rho\in S_P$ induces an automorphism $g_\rho:\overline{\mathcal M}_P\to \overline{\mathcal M}_P$, which maps the compactification divisor of $\mathcal M_P$ to itself (since $g_\rho$ extends an automorphism of $\mathcal M_P$). We abuse notation and denote the corresponding automorphism of $X^N\to X^N$ as $g_\rho:X^N\to X^N$. 

For any $\rho\in S_P$, define the map $f_\rho:=g_\rho\circ s:X^N\dashrightarrow
X^N$.   This map also has indeterminacy locus of dimension $N-2$,
since $I_{g_\rho \circ s} = I_s$. We now prove that the maps
$f_\rho:X^N\dashrightarrow X^N$ are algebraically stable.

%% file: stability.tex
\section{Algebraic Stability}\label{SEC:STABILITY}

The goal of this section is to prove the following proposition, which will imply Theorem \ref{THM_FUNCTORIAL_SEMIGROUP} and will be used to compute the linear maps $(f_\rho)^*:H^{k,k}(X^N;\C)\to H^{k,k}(X^N;\C)$.

\begin{proposition}\label{goodresolution} For any $N \geq 1$ there is an $N$-dimensional projective manifold $Y^N$ and holomorphic maps $\pr : Y^N \rightarrow X^N$ and $\widetilde{s}: Y^N \rightarrow X^N$
that make the following diagram commute (wherever $s \circ \pr$ is defined),
\begin{eqnarray}\label{EQN:DESIRED_DIAGRAM}
\xymatrix{Y^N \ar[d]^{\pr} \ar[dr]^{\widetilde{s}}  & \\
X^N \ar @{-->}[r]^{s} & X^N }
\end{eqnarray}
with $\widetilde{s}^{-1}(x)$  a finite set for every $x \in X^N$.
\end{proposition}

\noindent
{\em Proof of Theorem \ref{THM_FUNCTORIAL_SEMIGROUP}, supposing Proposition \ref{goodresolution}:}
Using the factorization $f_\rho = g_\rho \circ {s}$ with $g_\rho$ an automorphism of $X^N$, we obtain the following diagram:
\begin{eqnarray*}
\xymatrix{
Y^N \ar[d]^{\pr} \ar[dr]^{\widetilde{s}} \ar@/^1pc/[drr]^{\widetilde{f}_\rho}  &  & \\
X^N \ar @{-->}[r]^{s} & X^N \ar[r]^{g_\rho} & X^N.
}
\end{eqnarray*}
Since $\widetilde{s}$ has finite fibers, so does $\widetilde{f}_\rho:= g_\rho \circ \widetilde{s}$.
It follows from Proposition \ref{PROP:FUNCTORIAL} that for any rational map $h:X^N \dashrightarrow X^N$
we have $(h \circ f_\rho)^* = f_\rho^* \circ h^*$ on all cohomology groups.  By induction it follows that
for any word $f_{\rho_k} \circ \cdots \circ f_{\rho_1}$ we have
\begin{eqnarray*}
(f_{\rho_k} \circ \cdots \circ f_{\rho_1})^* = f_{\rho_1}^* \circ \cdots \circ f_{\rho_k}^*
\end{eqnarray*}
on $H^{k,k}(X^N;\C)$ for $1\leq k \leq N$.
This implies that the semi-group $\mathcal{F}^N$ acts functorially on all of the cohomology groups of $X^N$.
\qed

%Combining Propositions \ref{PROP:FUNCTORIAL} and \ref{goodresolution} also yields the following corollary, which
%will be helpful when computing $f_\rho^*$ in Section \ref{SEC:DYN_DEG}.
%
%\begin{corollary}\label{compcor}
%For any $n \geq 3$ and any $\rho \in S_n$, we have $f_\rho^* = (g_\rho \circ s)^* = s^* \circ g_\rho^*$.
%\end{corollary}

In order to prove Proposition \ref{goodresolution}, we will use the {universal property of blow ups}, following the treatment in \cite{EISENBUD_HARRIS,GW_Book}.
Let $X$ be any scheme and $Y \subseteq X$ a subscheme.  Recall that $Y$ is a Cartier subscheme if it is locally the zero locus of a single regular
function. 

{\bf Universal Property.} Let $X$ be a scheme and let $Y$ be a closed subscheme.  { The blow up of $X$ along $Y$} is a scheme $\widetilde{X}\equiv {\rm BL}_Y(X)$ and a morphism $\pi: \widetilde{X} \rightarrow X$
such that $\pi^{-1}(Y)$ is a Cartier subscheme and which is universal with respect to this property: if $\pi': \widetilde{X}' \rightarrow X$
is any morphism such that $(\pi')^{-1}(Y)$ is a Cartier subscheme, then there is a unique morphism $g: \widetilde{X}' \rightarrow \widetilde{X}$ such
that $\pi' = \pi \circ g$.

Recall that the Cartier subscheme $E = \pi^{-1}(Y)$ is called the {exceptional divisor} of the blow up and $Y$ is called the center of the blow up.

There is an immediate corollary of the definition, see for example \cite[Prop.~13.91]{GW_Book}:

\begin{corollary}\label{COR:BLOW_UP_PULL_BACK}
Let $X$ be a scheme, let $Y$ be a closed subscheme, and let $\pi:~{\rm BL}_Y(X)~\rightarrow~X$ be the blow up of $X$ along $Y$.  Let $f:~X'~\rightarrow~X$ be
any morphism of schemes.  Then, there exists a unique morphism ${\rm BL}_Y(f):~{\rm BL}_{f^{-1}(Y)}(X')~\rightarrow~{\rm BL}_Y(X)$ making the following
diagram commute:
\begin{eqnarray*}
\xymatrix{
 {\rm BL}_{f^{-1}(Y)}(X') \ar[rr]^{{\rm BL}_Y(f)} \ar[d] & & {\rm BL}_Y(X) \ar[d]  \\
X' \ar[rr]^{f} & & X
}
\end{eqnarray*}
\end{corollary}

In our context, $X$ will be an projective manifold and $Y \subseteq X$ will
be an projective submanifold.  
%Notice also that
%since $X$ and $Y$ are both manifolds, the blow-up ${\rm BL}_Y(X)$ will also be
%a projective manifold; see, for example, \cite{GH}.

The following is well-known, but we include a proof for completeness.

\begin{lemma}\label{LEM:COMMUTING_BLOW_UPS}Suppose that $X$ is an projective manifold and $Y,Z \subseteq X$ are projective submanifolds that intersect transversally (i.e. $T_x Y + T_x Z = T_x X$ at any $x \in Y \cap Z$).  Then,
\begin{itemize}
\item[(1)]  If $\pi: {\rm BL}_Y(X) \rightarrow X$ is the blow up, the proper transform $\widetilde{Z} = \overline{\pi^{-1}(Z \setminus Y)}$ and total transform $\pi^{-1}(Z)$ coincide.
\item[(2)]  ${\rm BL}_{\widetilde{Z}}({\rm BL}_Y(X)) \cong {\rm BL}_{\widetilde{Y}}({\rm BL}_Z(X))$.
\end{itemize}
\end{lemma}

\begin{proof}  Since the blow up along a submanifold is a local construction,  it suffices to check this statement
when $X = \C^N$,  $Y = {\rm span}(e_1,\ldots,e_k)$, and $Z = {\rm span}(e_l,\ldots,e_N)$, where $e_1,\ldots,e_N$ are the standard basis vectors in $\C^N$.  Since $Y$ and $Z$ are assumed transverse, $l \leq k+1$.
We have
\begin{eqnarray*}
{\rm BL}_Y(X) =  \{(x_1,\ldots,x_N) \times [m_{k+1}:\cdots:m_N] \in \C^N \times \P^{N-k-1} \, |\, \\
 (m_{k+1},\ldots,m_N) \sim  (x_{k+1},\ldots,x_N)\},
\end{eqnarray*}
where ${\bf u} \sim {\bf v}$ means that one vector is a scalar multiple of the other.
Notice that 
\begin{eqnarray*}
\widetilde{Z} = \{(0,\ldots,0,x_{l},\ldots,x_N) \times [m_{k+1}:\cdots:m_N]  \in \C^N \times \P^{N-k-1} \, | \\ \qquad \qquad (m_{k+1},\ldots,m_N) \sim  (x_{k+1},\ldots,x_N)\},
\end{eqnarray*}
which coincides with $\pi^{-1}(Z)$, proving (1).

If we blow up $\widetilde{Z}$, we find
\begin{align*}
{\rm BL}_{\widetilde Z}({\rm BL}_Y(X)) = \hspace{3.5in}\\
\{(x_1,\ldots,x_N) \times [m_{k+1}:\cdots:m_N] \times [n_1:\cdots:n_{l-1}]  \in \C^N \times \P^{N-k-1} \times \P^{l-2} \, | \\ \, (m_{k+1},\ldots,m_N) \sim (x_{k+1},\ldots,x_N) \mbox{ and } (n_1,\ldots,n_{l-1}) \sim (x_1,\ldots,x_{l-1})\}
\end{align*}
This is clearly isomorphic to the result we would obtain if we had first blown up $Z$ and then blown up $\widetilde{Y}$, proving (2).
\end{proof}

We will break the proof of Proposition \ref{goodresolution} into Lemmas \ref{existenceofY} and \ref{finitefibers}, below.
In order to keep notation as simple as possible, we will usually drop the dimension  $N$ from the notation, writing $X \equiv X^N$ and $Y \equiv Y^N$.

In order to construct $Y$, we first recall the construction of $X$. 
Recall 
\begin{eqnarray*}
A_0 = \{[1:0:\cdots:0],[0:1:0:\cdots:0],\ldots,[0:\cdots:0:1],[1:1:\cdots:1]\} \subseteq \P^{N},
\end{eqnarray*}
and that for any $1 \leq i \leq N-2$,  $A_i$ is the set of all $\binom{N+2}{i+1}$ linear subspaces of dimension $i$ spanned by $i+1$ distinct points
from $A_0$.  
The space $X$ was constructed as an iterated blow up 
\begin{eqnarray*}
X := X_{N-1} \xrightarrow{\alpha_{N-2}} X_{N-2} \xrightarrow{\alpha_{N-3}} \cdots \xrightarrow{\alpha_{2}} X_2 \xrightarrow{\alpha_1} X_1 \xrightarrow{\alpha_0} X_0,
\end{eqnarray*}
where $X_0 = \P^{N}$, and for each $0 \leq i \leq N-2$ we have that $\alpha_i: X_{i+1} \rightarrow X_i$ is the blow up of $X_i$ along
the proper transform $\widetilde{A_i}$ of $A_i$ under $\alpha_{0} \circ \cdots \circ \alpha_{i-1}$.

The following lemma helps keep track of intersections between centers of the blow ups.  It is a restatement of \cite[Lemma 3.2.3]{LLOYD_PHILLIPPS}.

\begin{lemma}\label{LEM:GOODCENTERS1} Let $L$ and $M$ be irreducible components of $A_i$ and $A_j$.  Then either: 
\begin{itemize}
\item[(1)] $L \cap M = \emptyset$, 
\item[(2)] $L \cap M$ is an irreducible component of $A_k$ for some $k \leq {\rm min}(i,j)$, or
\item[(3)] $L$ intersects $M$ transversally.
\end{itemize}
\end{lemma}

In particular, if $L$ and $M$ are two distinct centers of the same dimension and (1) or
(2) holds, then by the time we blow them up they will be disjoint.  Otherwise, (3) holds and (since transversality is preserved
under blowing up) the order of blow ups will not matter, by Lemma \ref{LEM:COMMUTING_BLOW_UPS}.

Let $Y_0 = X_0 = \P^N$ and let $s_0: Y_0 \rightarrow X_0$ be the squaring map
$$s_0:~[x_1:\cdots:x_{N+1}]~\mapsto~[x_1^2:\cdots:x_{N+1}^2].$$  For each $0
\leq i \leq N-2$, let $B_i:= s_0^{-1}(A_i)$, and let $C_i:=B_i-A_i$. The space
$Y$ is constructed by the following sequence of blow ups

\begin{eqnarray*}
Y:= Y_{N-1} \xrightarrow{\beta_{N-2}} Y_{N-2} \xrightarrow{\beta_{N-3}} \cdots \xrightarrow{\beta_{2}} Y_2 \xrightarrow{\beta_1} Y_1 \xrightarrow{\beta_0} Y_0
\end{eqnarray*}
where  $\beta_i: Y_{i+1} \rightarrow Y_i$ is the blow up of $Y_i$ along
the proper transform $\widehat{B_i}$ of $B_i$ under $\beta_{0} \circ \cdots \circ \beta_{i-1}$.

Since Properties (1), (2), and (3) of Lemma \ref{LEM:GOODCENTERS1} persist under taking inverse images by $s_0$, we also have the following lemma. 

\begin{lemma}\label{LEM:GOODCENTERS2} Let $L$ and $M$ be irreducible components of $B_i$ and $B_j$.  Then either: 
\begin{itemize}
\item[(1)] $L \cap M = \emptyset$,
\item[(2)] $L \cap M$ is an irreducible component of $B_k$ for some $k \leq {\rm min}(i,j)$, or
\item[(3)] $L$ intersects $M$ transversally.
\end{itemize}
\end{lemma}

\begin{lemma}\label{existenceofY}
There exist maps $\pr: Y \rightarrow X$ and $\widetilde{s}: Y \rightarrow X$ making Diagram (\ref{EQN:DESIRED_DIAGRAM}) commute (where $s \circ \pr$ is defined).
\end{lemma}

\begin{proof}
Consider the diagram
\begin{eqnarray}\label{EQN:LADDER1}
\xymatrix{
Y_{N-1} \ar[d]^{\beta_{N-2}}\ar[r]^{\pr_{N-1}} & X_{N-1} \ar[d]^{\alpha_{N-2}} \\
Y_{N-2}\ar[r]^{\pr_{N-2}} & X_{N-2} \\
%\vdots & \vdots \\
%Y_{i+1} \ar[d]^{\beta_{i}} \ar[r]^{\pr_{i+1}} & X_{i+1} \ar[d]^{\alpha_{i}} \\
%Y_{i} \ar[r]^{\pr_{i}} & X_{i} \\
\vdots & \vdots \\
Y_1 \ar[d]^{\beta_0} \ar[r]^{\pr_1} & X_1 \ar[d]^{\alpha_0} \\
Y_0  \ar[r]^{\pr_0 = {\rm id}} &  X_0
}
\end{eqnarray}
We will use induction to prove that for every $0 \leq i \leq N-1$ there are mappings $\pr_i: Y_i
\rightarrow X_i$ making the diagram commute
with the following two additional properties:
\begin{enumerate}%rewritten by sarah
\item  for any $i \leq l \leq N-2$ we have that
\[
D_l:=\overline{\pr_i^{-1}\left(\widetilde{A}_l \right) \setminus \widehat{A}_l}
\]
is a Cartier subscheme of $Y_i$,
where tilde denotes proper transform under $\alpha_{i-1} \circ \cdots \circ \alpha_0$ and hat denotes proper transform
under $\beta_{i-1} \circ \cdots \circ \beta_0$, and
\item for every $i \leq l, m \leq N-2$ we have $\widehat{A}_l \not \subseteq D_m$.
\end{enumerate}

\begin{remark}
When $N = 2$ and, therefore $i = 0$, $D_l = \emptyset$.  However, $D_l$ is typically nonempty, 
including the case of $N = 3$ and $i = 1$, where
\begin{eqnarray*}
D_1 = E_{[1:-1:1:1]} \cup E_{[1:1:-1:1]} \cup E_{[1:1:1:-1]}.
\end{eqnarray*}
These are the exceptional divisors obtained when blowing up the points of $B_0 \cap A_1$.
In general, $D_l$ can be thought of as the ``extra'' exceptional divisors lying over ${A}_l$ produced in the construction of $Y_i$ that were not producted
in the construction of~$X_i$.
\end{remark}

\vspace{0.1in}
As the base-case of the induction, notice that $\pr_0 = {\rm id}: Y_0
\rightarrow X_0$ trivially satisfies both (1) and (2).

We now suppose that there is a mapping $\pr_i: Y_i \rightarrow X_i$ 
for which Properties (1) and (2) hold.  We'll
use the universal property of blow ups to construct $\pr_{i+1}: Y_{i+1}
\rightarrow X_{i+1}$ for which  Properties (1) and (2)  hold as well.

By Lemmas \ref{LEM:GOODCENTERS2} and \ref{LEM:COMMUTING_BLOW_UPS} we can perform the blow ups of
irreducible components of $\widehat B_{i}$ in any order we like; recall that $C_i:=B_i \setminus A_i$.
Let us first blow up $\widehat A_{i}$ and then $\widehat C_{i}$,
factoring $\beta_{i}$ as a composition $Y_{i+1} \xrightarrow{\mu_{i}} Z_{i+1}
\xrightarrow{\lambda_{i}} Y_i$, where $\lambda_{i}$ is the blow up along
$\widehat{A}_{i}$ and $\mu_{i}$ is the further blow up along along
$\widehat{C}_{i}$.  Let $\eta_{i}: =
\pr_i \circ \lambda_{i}$ and consider the following diagram. 

\begin{eqnarray}
\xymatrix{
Y_{i+1} \ar[d]^{\mu_{i}} \ar@/_2pc/@{->}[dd]_{\beta_{i}} & \\
Z_{i+1} \ar[d]^{\lambda_{i}} \ar[dr]^{\eta_{i}} & X_{i+1} \ar[d]^{\alpha_{i}} \\
Y_{i} \ar[r]^{\pr_i} & X_i
}
\end{eqnarray}
We will use the universal property to construct $q_{i+1}: Z_{i+1} \rightarrow X_{i+1}$ making the diagram commute.  Then, $\pr_{i+1}: = q_{i+1} \circ \mu_{i}$ will be the desired map.
\begin{eqnarray}\label{EQN:INDUCTION_STEP}
\xymatrix{
Y_{i+1} \ar[d]^{\mu_{i}} \ar[dr]^{\pr_{i+1}} \ar@/_2pc/@{->}[dd]_{\beta_{i}} & \\
Z_{i+1} \ar[d]^{\lambda_{i}} \ar[dr]^{\eta_{i}} \ar[r]^{q_{i+1}}& X_{i+1} \ar[d]^{\alpha_{i}} \\
Y_{i} \ar[r]^{\pr_i} & X_i
}
\end{eqnarray}
By the induction hypothesis, $\pr_i^{-1}(\widetilde{A}_i) = \widehat{A}_i \cup D_i$, where $D_i$ is an Cartier subscheme.
%Since $\lambda_{i}: Z_{i+1} \rightarrow Y_i$ is the blow-up along $\widehat{A}_{i}$, $\lambda_{i}^{-1}\left(\widehat{A}_{i}\right)$
%is the exceptional divisor $E_{\widehat{A}_i}$, in particular it is a Cartier subscheme.  Meanwhile, $\lambda_{i}^{-1}(D_i)$
%is a Cartier subscheme, since Cartier subschemes are preserved under  pullback by regular maps.  Thus,
By Property (1) of the induction hypothesis
\begin{eqnarray*}
\eta_{i}^{-1}(\widetilde{A}_i) = \lambda_{i}^{-1}(\pr_i^{-1}(\widetilde{A}_i)) = \lambda_{i}^{-1}(\widehat{A}_i \cup D_i) = E_{\widehat{A}_i} \cup \lambda_{i}^{-1}(D_i)
\end{eqnarray*}
is a Cartier subscheme (where $E_{\widehat{A}_i}$ denotes the exceptional divisor).  By the universal property of blow ups, there exists a map  $q_{i+1} : Z_{i+1} \rightarrow X_{i+1}$ making the diagram commute.

We now must check that $\pr_{i+1}:= q_{i+1} \circ \mu_{i}$ satisfies
Properties (1) and (2).  We'll first show that $q_{i+1}$ satisfies the
these properties.  We will continue to use tildes to denote proper transforms living in $X_i$.  When taking
a further proper transform under $\alpha_{i}$, we will append $'$.  Similarly, we will continue to use hats
to denote proper transforms living in $Y_i$ and we'll append  $'$ to denote a further proper transform under
$\lambda_{i}$ and  $''$ to denote a further proper transform under $\mu_{i}$.

Suppose $i+1 \leq l \leq N-2$.  Consider the proper transform of
$\widetilde{A}_l$ under $\alpha_{i}$, which is given by $\widetilde{A}_l'~=~\overline{\alpha_{i}^{-1}(\widetilde{A}_l \setminus \widetilde{A}_{i})}$.
Since $q_{i+1}: Z_{i+1} \rightarrow X_{i+1}$ is continuous and closed, we have 
\begin{eqnarray*}
(q_{i+1})^{-1}\left(\widetilde{A}_l'\right) = (q_{i+1})^{-1}\bigg(\overline{\alpha_{i}^{-1}(\widetilde{A}_l \setminus \widetilde{A}_{i})}\bigg) &=& \overline{(\alpha_{i} \circ q_{i+1})^{-1}(\widetilde{A}_l \setminus \widetilde{A}_{i})} \\
&=&  \overline{\lambda_{i}^{-1} \circ \pr_{i}^{-1} (\widetilde{A}_l \setminus \widetilde{A}_{i})},
\end{eqnarray*}
using commutativity of (\ref{EQN:INDUCTION_STEP}).
By the induction hypothesis, $\pr_{i}^{-1}(\widetilde{A}_l) = \widehat{A}_l \cup D_l$ and $\pr_{i}^{-1}(\widetilde{A}_{i}) = \widehat{A}_{i} \cup D_{i}$ with $D_{i}$ and $D_l$ both Cartier subschemes and $A_{l} \not \subseteq D_{i}$.
We have
\begin{eqnarray*}
\pr_{i}^{-1} (\widetilde{A}_l \setminus \widetilde{A}_{i}) = \left(\widehat{A}_l \cup D_l\right) \setminus \left(\widehat{A}_{i} \cup D_{i}\right) = \left(\widehat{A}_l \setminus (\widehat{A}_{i} \cup D_{i})\right) \bigcup \left((D_l \setminus D_{i}) \setminus \widehat{A}_{i}\right).
\end{eqnarray*}
Since $\widehat{A}_l \not \subseteq D_{i}$, we have that
\begin{eqnarray*}
\overline{\lambda_{i}^{-1}(\widehat{A}_l \setminus (\widehat{A}_{i} \cup D_{i}))} = \widehat{A}_l'.
\end{eqnarray*}
Meanwhile, since $D_l$ and $D_{i}$ are Cartier subschemes of $Y_i$
\begin{eqnarray*}
H_l := \overline{\lambda_{i}^{-1}((D_l \setminus D_{i}) \setminus \widehat{A}_{i})}\subseteq Z_{i+1}
\end{eqnarray*}
is a (potentially empty) Cartier subscheme.
Thus,
\begin{eqnarray*}
(q_{i+1})^{-1}\left(\widetilde{A}_l'\right) = \overline{\lambda_{i+1}^{-1} \circ \pr_{i}^{-1} (\widetilde{A_l} \setminus \widetilde{A}_{i+1})} = \widehat{A}_l' \cup H_l.
\end{eqnarray*}

By the induction hypothesis, we have that for all $i+1 \leq l, m \leq N-2$,
$\widehat{A}_l
\not \subseteq D_m$ so that $\widehat{A}_l \cap D_m$ is a proper subvariety of
$\widehat{A}_l$.  Since $\widehat{A}_{i}$ is of lower dimension than
$\widehat{A}_l$, there is a point $y \in \widehat{A}_l \setminus
(\widehat{A}_{i} \cup D_m)$.  Since $\lambda_{i}$ is surjective, any
element of $\lambda_{i}^{-1}(y)$ gives a point of $\widehat{A}_l'
\setminus H_m$.  Thus, $\widehat{A}_l'
\not \subseteq H_m$.

We will now pull everything back via the total transform under 
$\mu_{i}$ and check that Properties (1) and~(2) hold for $\pr_{i+1}: = q_{i+1} \circ \mu_{i}$.  
Consider any $i+1 \leq l \leq N-2$.
It follows from Lemma \ref{LEM:GOODCENTERS1} that for any irreducible components $L$ of
$\widehat{C}_{i}'$ and $M$ of $\widehat{A}_l'$ we have either $L \cap M = \emptyset$, $L$ and $M$ are transverse, or  $L \subseteq M$.
In the first case, the total transform of $M$ under the blow up of $L$ coincides with the proper transform $M''$.  This also
holds in the second case, by Lemma \ref{LEM:COMMUTING_BLOW_UPS}.  In
the last case, the total transform of $M$ is $M' \cup E_{L}$, where $E_{L}$ is the exceptional divisor over $L$.
Therefore,
\begin{eqnarray*}
\mu_{i+1}^{-1}(\widehat{A}_l') = \widehat{A}_l'' \cup E_l,
\end{eqnarray*}
where $E_l$ is the union of exceptional divisors over the components of $\widehat{C}_{i}'$ lying entirely within $\widehat{A}_l'$.
Meanwhile 
\begin{eqnarray*}
K_l := \mu_{i}^{-1}(H_l)
\end{eqnarray*}
is a Cartier subscheme. Thus, 
\begin{eqnarray*}
\pr_{i+1}^{-1}(\widetilde{A}'_{l}) = \mu_{i}^{-1}(\widehat{A}_l' \cup H_l) = \widehat{A}_l'' \cup E_l  \cup K_l
\end{eqnarray*}
where $E_l \cup K_l$ is a Cartier subscheme.  In particular, Property (1) holds.

To see that Property (2) holds, notice that for any $i+1 \leq l,m \leq N-2$ we have $\widehat{A}_l'' \not \subseteq E_{\widehat{C}_{i}}$
since $\widehat{A}_l$ is of greater dimension than $\widehat{C}_{i}'$.  Taking a point $y \in \widehat{A}_l \setminus (\widehat{C}_{i}' 
\cup H_m)$, we see that $\mu_{i}^{-1}(y)$ is a nonempty subset of $\widehat{A}_l'' \setminus (E_m \cup K_m)$.
Thus, $\widehat{A}_l''  \not \subseteq (E_m \cup K_m)$ establishing that Property (2) holds.

By induction, we conclude that for each $0 \leq i \leq N-1$ there exist mappings $\pr_i: Y_i \rightarrow X_i$ making Diagram \ref{EQN:LADDER1} commute.

\vspace{0.1in}

We'll now construct the map $\widetilde{s} : Y \rightarrow X$.  Let $s_0: \P^N \rightarrow \P^N$ be the squaring map.
Since $B_0 = s_0^{-1}(A_0)$, Corollary \ref{COR:BLOW_UP_PULL_BACK} gives that $s_0 \equiv \widetilde{s}_0 : Y_0 \rightarrow X_0$ lifts to a holomorphic map $\widetilde{s}_1: Y_1 \rightarrow X_1$:
\begin{eqnarray*}
\xymatrix{
Y_1 \ar[r]^{\widetilde{s}_1} \ar[d]^{\beta_0} & X_1 \ar[d]^{\alpha_0} \\
Y_0 \ar[r]^{\widetilde{s}_0} & X_0
}
\end{eqnarray*}\noindent
Notice that $\widehat{B}_1 = (\widetilde{s}_1)^{-1}(\widetilde{A}_1)$, so that we can
again apply Corollary \ref{COR:BLOW_UP_PULL_BACK} to lift $\widetilde{s}_1$ to a
holomorphic map $\widetilde{s}_2: Y_2 \rightarrow X_2$ making the following diagram
commute:

\begin{eqnarray}\label{EQN:PARTIAL_DIAGRAM_FOR_STILDE}
\xymatrix{
Y_2 \ar[d]^{\beta_1} \ar[r]^{\widetilde{s}_2} & X_2 \ar[d]^{\alpha_1} \\
Y_1 \ar[d]^{\beta_0} \ar[r]^{\widetilde{s}_1} & X_1 \ar[d]^{\alpha_0} \\
Y_0 \ar[r]^{\widetilde{s}_0} & X_0
}
\end{eqnarray}

Continuing in this way, we obtain holomorphic map $\widetilde{s}_i: Y_i \rightarrow X_i$ for $1 \leq i \leq N-1$ making the following diagram commute:

\begin{eqnarray}\label{EQN:DIAGRAM_FOR_STILDE}
\xymatrix{
Y_{N-1} \ar[r]^{\widetilde{s}_{N-1}} \ar[d]^{\beta_{N-2}} & X_{N-1} \ar[d]^{\alpha_{N-2}} \\
Y_{k-2} \ar[r]^{\widetilde{s}_{N-2}}  & X_{N-2}  \\
\vdots & \vdots \\
%Y_2 \ar[d]^{\beta_1} \ar[r]^{\widetilde{s}_2} & X_2 \ar[d]^{\alpha_1} \\
Y_1 \ar[d]^{\beta_0} \ar[r]^{\widetilde{s}_1} & X_1 \ar[d]^{\alpha_0} \\
Y_0 \ar[r]^{\widetilde{s}_0} & X_0
}
\end{eqnarray}
The desired map is $\widetilde{s} \equiv \widetilde{s}_{N-1}: Y_{N-1} \rightarrow X_{N-1}$. 

\vspace{0.1in}
We must now check that Diagram (\ref{EQN:DESIRED_DIAGRAM}) commutes wherever $s \circ \pr$ is defined, i.e.
on $Y \setminus \pr^{-1}(I_s)$.  Since $Y$ is connected, 
it suffices to prove commutativity on
any open subset of $Y \setminus \pr^{-1}(I_s)$.  
Let
\begin{eqnarray}\label{EQN:DEF_AA_BB}
\AA := \alpha_0 \circ \cdots \circ \alpha_{N-2}: X \rightarrow \P^N \qquad \mbox{and} \qquad \BB = \beta_0  \circ \cdots \circ \beta_{N-2}: Y \rightarrow \P^N
\end{eqnarray}  be the compositions of the blow ups
used to construct $X$ and $Y$.
Consider an open subset  $U \subseteq \P^N$
with $U$ disjoint from $\cup_{i=0}^{N-2} B_i$ 
Then, $\BB|_{{\BB}^{-1}(U)} : \BB^{-1}(U) \rightarrow U$ and $\AA|_{\AA^{-1}(U)} : \AA^{-1}(U) \rightarrow U$
serve as local coordinate charts on $Y$ and $X$.  Commutativity of (\ref{EQN:LADDER1}) gives that when $\pr$ is expressed
in these coordinates it becomes the identity. 

Since $V := \widetilde{s}_0(U)$ is disjoint from $\cup_{i=0}^{N-2} A_i$, we have that
$\AA |_{\AA ^{-1}(V)} : \AA^{-1}(V) \rightarrow V$
serves as a local coordinate chart on $X$.  Commutativity of
(\ref{EQN:DIAGRAM_FOR_STILDE}) implies that when expressed in the $\BB |_{\BB^{-1}(U)}$ and $\AA|_{\AA^{-1}(V)}$
coordinates, $\widetilde{s}$ is given by $\widetilde{s}_0 : U \rightarrow V$. 

By definition, when  $s: X \dashrightarrow X$ is expressed in the 
$\BB|_{\BB^{-1}(U)}$ and $\AA|_{\AA^{-1}(V)}$ coordinates, it becomes $\widetilde{s}_0 : U \rightarrow V$.  
Therefore, when expressed in the $\BB|_{\BB^{-1}(U)}$ and $\AA|_{\AA^{-1}(V)}$ coordinates
$s \circ \pr$ is also given by  $\widetilde{s}_0 : U \rightarrow V$.
We conclude that (\ref{EQN:DESIRED_DIAGRAM}) commutes wherever $s \circ \pr$ is defined.
\end{proof}

\begin{lemma}\label{finitefibers}
Let $\widetilde s:Y\to X$ be the map constructed above. For every $x \in X$ the set $\widetilde{s}^{-1}(x)$ is finite.
\end{lemma}

\noindent
The proof of this lemma was inspired by techniques of Lloyd-Philipps \cite{LLOYD_PHILLIPPS}.

{\em Proof.} The proof will proceed by induction on the dimension $N$.  
In addition to using superscripts to index the dimension of the spaces $X^N$ and $Y^N$, we'll also 
occasionally append them to our maps in order to specify the dimension of the spaces in the domain and codomain of the maps. For example, the superscript on $\widetilde{s}^N$ indicates that it is a mapping $\widetilde{s}^N: Y^N \rightarrow X^N$and the superscript on $A_l^N$ indicates that it's a subset of $\P^N$.

For the inductive proof, it will be helpful to consider the one-point spaces $\P^0, X^0,$ and $Y^0$ for which it's trivial that $\widetilde{s}^0: Y^0 \rightarrow X^0$ has finite fibers.
When $N =1$ we have  $Y^1 \equiv Y^1_0 \equiv \P^1$ and $X^1 \equiv X^1_0
\equiv \P^1$ and $\widetilde{s}^1 \equiv \widetilde{s}^1_0 : \P^1 \rightarrow
\P^1$ is the squaring map, which clearly has finite fibers.

Now, suppose that for each $1 \leq i < N$, the mappings $\widetilde{s}^i: Y^i
\rightarrow X^i$ have finite fibers in order to prove that 
$\widetilde{s}^N :  Y^N \rightarrow X^N$ has finite fibers.

Recall the commutative diagram:
\begin{eqnarray}\label{EQN:ANOTHER_DIAGRAM_FOR_STILDE}
\xymatrix{
Y \ar[r]^{\widetilde{s}} \ar[d]^{\BB} & X \ar[d]^{\AA} \\
\P^N \ar[r]^{\widetilde{s}_0} & \P^N
}
\end{eqnarray}
where $\AA$ and $\BB$ are the compositions of blow ups defined in
(\ref{EQN:DEF_AA_BB}).  Let $z = \AA(x)$ and notice that since the squaring map
$\widetilde{s}_0$ has finite fibers, there are finitely many points $w \in
\widetilde{s}_0^{-1}(z) \subset \P^N$ over which the preimages $\widetilde{s}^{-1}(x)$ lie.
Thus, for any such $z$ and $w$ it suffices show that
\begin{eqnarray}\label{EQN:RESTRICTED_MAP}
\widetilde{s} | _{\BB^{-1}(w)} : \BB^{-1}(w) \rightarrow \AA^{-1}(z)
\end{eqnarray}
has finite fibers.  

If $\AA(x)$ is not a critical value of $\widetilde{s_0}$, 
then for any $y \in \tilde{s}^{-1}(x)$ there is a neighborhood
$U$ of $\BB(y)$ so that $\widetilde{s}_0: U \rightarrow \widetilde {s}_0(U) =: V$ is a biholomorphism.  Iteratively
applying Corollary \ref{COR:BLOW_UP_PULL_BACK} to $\widetilde{s}_0$ and its inverse gives
that $\widetilde{s} : \BB^{-1}(U) \rightarrow \AA^{-1}(V)$ is a biholomorphism.

If $\AA(x)$ is a critical value, the proof is more subtle.
We will use the recursive structure of $X^N$ and $Y^N$ in order to express these fibers as products 
of lower dimensional $X^i$ and $Y^i$, which will allow us to express
(\ref{EQN:RESTRICTED_MAP}) in terms of $\widetilde{s}^i: Y^i \rightarrow X^i$ and ${\rm id}^i: X^i \rightarrow X^i$, for $0 \leq i < N$, where ${\rm id}^i$ is the identity mapping.

Notice that the construction of $\widetilde{s}: Y \rightarrow X$ commutes
with permutations 
\begin{eqnarray*}
\sigma:[z_1:z_2:\cdots:z_{N+1}] \mapsto [z_{\sigma(1)}:z_{\sigma(2)}:\cdots:z_{\sigma(N+1)}]
\end{eqnarray*}
 of the coordinates on $\P^N$.  In particular, we can suppose without loss of generality that 
$z = \AA(x) = [0:\cdots:0:z_{l+1}:\cdots:z_{N+1}]$ with $z_i \neq 0$ for $l+1 \leq i \leq N+1$ and that the remaining $z_i$ are grouped so that
repeated values come in blocks.  (Note that one can have $l = 0$.)
Commutative Diagram (\ref{EQN:ANOTHER_DIAGRAM_FOR_STILDE}) 
implies that  $w = [0:\cdots:0:w_{l+1}:\cdots:w_{N+1}]$, with $w_i
\neq 0$ for $l+1 \leq i \leq N+1$.

We need a more precise description of centers of the blow ups $A_l^N$.
Let
\begin{eqnarray*}
q_1 = [1:0:\cdots:0],\ldots,q_{N+1} = [0:\cdots:0:1],q_{N+2} = [1:1:\cdots:1] \, \in \P^N
\end{eqnarray*}
and for any $\{i_1,\ldots,i_{m+1}\} \subseteq \{1,\ldots,N+2\}$, let
\begin{eqnarray*}
\Pi_{i_1,\ldots,i_{m+1}} := {\rm span}(q_{i_1},\ldots,q_{i_{m+1}}) \subseteq \P^N.
\end{eqnarray*}
Note that
\begin{eqnarray*}
A_m = \bigcup_{\{i_1,\ldots,i_{m+1}\}} \Pi_{i_1,\ldots,i_{m+1}}
\end{eqnarray*}
where the union is taken over all subsets $\{i_1,\ldots,i_{m+1}\}\subseteq \{1,\ldots,N+2\}$.

\vspace{0.1in}
We will need a more precise description of which components of $z$ are equal. Let's suppose that
\begin{eqnarray*}
z_{l+1} &=& z_{l+2} = \cdots = z_{l+j_1}, \\
z_{l+j_1+1} &=& z_{l+j_1+2} = \cdots = z_{l+j_1+j_2}, \\
&\vdots&  \\
z_{l+j_1+\cdots + j_{a-1}+1} &=& z_{l+j_1+\cdots + j_{a-1}+2} = \cdots = z_{l+j_1+\cdots+j_a},
\end{eqnarray*}
with no equality between any pair of lines.
In other words, the first  $j_1$ nonzero entries are equal, the next $j_2$
nonzero entries are equal and distinct from the first $j_1$ nonzero entries,
etc...  We assume that each $j_1,\ldots,j_a \geq 2$ and that all values
appearing in the remaining components of $z$ occur only once. 

We'll show that 
\begin{eqnarray}\label{EQN:ZFIBER}
\AA^{-1}(z) \cong X^{l-1} \times X^{j_1-2} \times \cdots \times X^{j_a-2},
\end{eqnarray} 
%for any $w \in \widetilde s_0^{-1}(z)$ that
\begin{eqnarray}\label{EQN:WFIBER}
\BB^{-1}(w) \cong Y^{l-1} \times X^{j_1-2} \times \cdots \times X^{j_a-2},
\end{eqnarray}
and that (in the coordinates given by these isomorphisms)
\begin{eqnarray}\label{EQN:RESTRICTION}
\widetilde{s} | _{\BB^{-1}(w)} = \widetilde{s}^{\, l-1} \times {\rm id}^{j_1-2} \times \cdots {\rm id}^{j_a-2},
\end{eqnarray}
where ${\rm id}^i: X^i \rightarrow X^i$ denotes the identity mapping.
By the induction hypothesis, (\ref{EQN:RESTRICTION}) will imply that $\widetilde{s} | _{\BB^{-1}(w)} : {\BB^{-1}(w)} \rightarrow \AA^{-1}(z)$
has finite fibers and thus complete the proof.

We'll first check that (\ref{EQN:ZFIBER}) holds.  Let $V \subseteq \P^N$ be a
neighborhood of $z$ chosen small enough so that it intersects
$\Pi_{i_1,\ldots,i_{m+1}}$ if and only if $z \in \Pi_{i_1,\ldots,i_{m+1}}$.  In order to study $\AA^{-1}(z)$ we'll work with $\AA^{-1}(V)$.

Associated to the particular points $z\in\P^N$ above, we have the following sets. Let $S = \{1,\ldots,N+2\}$, and let 
\begin{eqnarray*}
S^0 &:=& S \setminus \{1,\ldots,l,N+2\}, \\
S^1 &:=& S \setminus \{l+1,\ldots,l+j_1\}, \\
S^2 &:=& S \setminus \{l+j_1+1,\ldots,l+j_1+j_2\}, \\
    & \vdots & \\
S^a &:=& S \setminus \{l+j_1+\cdots+j_{a-1},\ldots,l+j_1+\cdots+j_a\}.
\end{eqnarray*}
Note that $z \in \Pi_{i_1,\ldots,i_{m+1}}$ if and only if $S^b \subseteq
\{i_1,\ldots,i_{m+1}\}$ for some $0 \leq b \leq a$.  We will call each of the
centers $\Pi_{S^b}$ for $0 \leq b \leq a$ {\em primitive center} since any
center of blow up through $z$ will contain at least one of them.
Since $S_b \cup S_c = S$ for any $b \neq c$, any center through $z$ that is blown-up contains a unique primitive center.
Thus, any further center that is blown up is of the form
\begin{eqnarray*}
\Pi_T\quad\text{where}\quad T=S^b\cup \{i_0,\ldots,i_m\}.
%\Omega^b_{i_0,\ldots,i_m} := \Pi_{S^b \cup \{i_0,\ldots,i_m\}} \supseteq \Omega^b.
\end{eqnarray*}
We will call $\Pi_T$ {\em subordinate} to $\Pi_{S^b}$.

Since $S_b \cup S_c = S$ for $b \neq c$, it also follows that any center subordinate to $\Pi_{S^b}$ is transverse to any center subordinate to
$\Pi_{S^c}$.  Since blow ups preserve transversality, this will also hold for the proper transforms.
Therefore, 
by Lemma \ref{LEM:COMMUTING_BLOW_UPS}, we can exchange the order of blow up between two centers subordinate to
distinct primitive centers and still get the same result for $\AA^{-1}(V)$.
In particular, we can first blow up each of the primitive centers.  After doing
so, we can blow up all of the (proper transforms of) centers subordinate to
$\Pi_{S^0}$, by order of increasing dimension.  We can then blow up all (proper transforms of) centers
subordinate to $\Pi_{S^1}$ by order of increasing dimension, etc... 

Let $[v_1:v_2:\cdots:v_{N+1}]$ be homogeneous coordinates on $\P^N$.  Blowing up $\Pi_{S^0}$ produces
{\footnotesize
\begin{eqnarray*}
\{[v_1:v_2:\cdots:v_{N+1}] \times [m^0_1:\cdots:m^0_l] \in V \times \P^l \, : \, (v_1,\ldots,v_l) \sim (m^0_1,\ldots,m^0_l)\},
\end{eqnarray*}
}
\noindent
where $\sim$ indicates that one vector is a multiple of another.

Blowing-up each of the remaining primitive centers $\Pi_{S^1},\ldots,\Pi_{S^a}$ produces the subset of $V \times \P^{l-1} \times \P^{j_1-2} \times \cdots \times \P^{j_a-2}$ given in the coordinates
\begin{eqnarray*}
\{[v_1:\cdots:v_{N+1}] \times [m^0_1:\cdots:m^0_l] \times [m^1_1:\cdots:m^1_{j_1-1}] \times \cdots \times  [m^a_1:\cdots:m^a_{j_a-1}] \end{eqnarray*}
by the conditions
{\footnotesize
\begin{eqnarray*}
(m^0_1,\ldots,m^0_l) &\sim& (v_1,\ldots,v_l), \\
(m^1_1,\ldots,m^1_{j_1-1}) &\sim& (v_{l+2}-v_{l+1},v_{l+3}-v_{l+1},\ldots,v_{l+j_1}-v_{l+1}) \\
&\vdots& \hspace{1in} \\
(m^a_1,\ldots,m^a_{j_a-1}) &\sim& (v_{l+j_1+\cdots+j_{a-1}+2}-v_{l+j_1+\cdots+j_{a-1}+1},\ldots,v_{l+j_1+\cdots+j_a}-v_{l+j_1+\cdots+j_{a-1}+1}).
\end{eqnarray*} 
}

Let us denote this blow up at all of the primitive centers by $\nu: V^\# \rightarrow V$.
The fiber over $z$ is $\nu^{-1}(z) \cong \P^{l-1} \times \P^{j_1-2} \times \cdots \times \P^{j_a-2}$.  
We'll now check that blow ups along the proper transforms of the centers subordinate to $\Pi_{S^0},\ldots,\Pi_{S^a}$ result in suitable blow ups
of $\nu^{-1}(z)$ in order to transform it into $X^{l-1} \times X^{j_1-2}\times \cdots X^{j_a-2}$.

Each of the centers subordinate to $\Pi_{S^0}$ will be of the form $\Pi_T$ where $T=S^0\cup \{i_1,\ldots,i_m\}$, for $\{i_1,\ldots,i_m\} \subseteq \{1,\ldots,l,N+2\}$.
There are precisely $l+1$ centers of one dimension greater than the dimension of $\Pi_{S^0}$; they are 
\[
\Pi_{S^0 \cup\{1\}},\ldots,\Pi_{S^0\cup \{l\}},\Pi_{S^0\cup \{N+2\}}.
\]
One can check that the proper transforms
of these intersect $\P^{l-1} \times \P^{j_1-2} \times \cdots \times \P^{j_a-2}$ at
\begin{eqnarray}\label{EQN:INTERSECTIONS_FIRST_FACTOR}
&& \{[1:0:\cdots:0]\} \times \P^{j_1-2} \times \cdots \times \P^{j_a-2},\\
&& \hspace{1in} \vdots  \nonumber \\
&&\{[0:0:\cdots:1]\} \times \P^{j_1-2} \times \cdots \times \P^{j_a-2}, \nonumber \mbox{and} \\
&&\{[1:1:\cdots:1]\} \times \P^{j_1-2} \times \cdots \times \P^{j_a-2}, \nonumber
\end{eqnarray}
respectively.  In other words, the centers of dimension one greater than $\Pi_{S^0}$ that are subordinate to $\Pi_{S^0}$ 
intersect $\nu^{-1}(z)$ in $A_0^{l-1} \times \P^{j_1-2} \times \cdots \times \P^{j_a-2}$.

If we let 
\begin{eqnarray*}\hat q_1 = [1:0:\cdots:0],\ldots,\hat q_{l} = [0:0:\cdots:1], \hat q_{l+1} = [1:1:\cdots:1] \in \P^{l-1},
\end{eqnarray*}
then for any $\{i_1,\ldots,i_m\} \subseteq \{1,\ldots,l,N+2\}$, one can check that the proper transform of 
\[
\Pi_{S^0\cup\{i_1,\ldots,i_m\}}
\]
intersects $\nu^{-1}(z)$ in $\hat \Pi_{i_0,\ldots,i_m} \times \P^{j_1-2} \times \cdots \times \P^{j_a-2}$, where
\begin{eqnarray*}
\hat \Pi_{i_1,\ldots,i_m} = {\rm span}\{\hat q_{i_1},\ldots,\hat q_{i_m}\} \subseteq \P^{l-1}.
\end{eqnarray*}
In particular, for any $1 \leq b \leq l-1$, the centers of dimension $b$ greater than the dimension of $\Pi_{S^0}$ that are subordinate to $\Pi_{S^0}$
intersect $\nu^{-1}(z)$ in $A_{b-1}^{l-1} \times \P^{j_1-2} \times \cdots \times \P^{j_a-2}$.

Therefore, blowing up all of the centers subordinate to $\Pi_{S^0}$ in order of
increasing dimension results in a sequential blow up of the first factor $\P^{l-1}$
transforming it into $X^l$.  It leaves each of the remaining factors unchanged.

Matters are almost the same for the remaining factors.  Let us illustrate the only difference by discussing the second factor $\P^{j_1-2}$.
Each of the centers subordinate to $\Pi_{S^1}$ will be of the form 
\[
\Pi_{S^1\cup\{i_1,\ldots,i_m\}}\quad\text{where}\quad \{i_1,\ldots,i_m\} \subseteq \{l+1,\ldots,l+j_1\}.
\]
Thus there are $j_1$ centers of dimension one greater: 
\[
\Pi_{S^1\cup\{l+1\}},\ldots,\Pi_{S^1\cup\{l+j_1\}}.
\]
 One can check that their proper transforms intersect
$\nu^{-1}(z)$ in:
\begin{eqnarray*}
&& \P^{l-1} \times  \{[1:1:\cdots:1]\} \times \P^{j_2-2} \times \cdots \times \P^{j_a-2},\\
&& \P^{l-1} \times \{[1:0:\cdots:0]\} \times \P^{j_2-2} \times \cdots \times \P^{j_a-2}, \mbox{and} \\
&& \hspace{1in} \vdots  \\
&& \P^{l-1} \times \{[0:0:\cdots:1]\} \times \P^{j_2-2} \times \cdots \times \P^{j_a-2},
\end{eqnarray*}
respectively.  Using similar reasoning to that from the analysis of the first factor, we can see that the centers of dimension $b$ greater than $\Pi_{S^1}$ that are subordinate to $\Pi_{S^1}$ will intersect $\nu^{-1}(z)$ in $\P^{l-1} \times A_{b-1}^{j_1-2} \times  \P^{j_1-2} \times \cdots \times \P^{j_a-2}$.  In particular, blowing up all centers subordinate to $\Pi_{S^1}$ in order of dimension will result in blowing up the second factor from $\P^{j_1-2}$ to $X^{j_1-2}$.

We conclude that (\ref{EQN:ZFIBER}) holds.

\vspace{0.1in}
We will now prove (\ref{EQN:WFIBER}) and (\ref{EQN:RESTRICTION}).
Let $U$ be the component of $\widetilde s_0^{-1}(V)$ containing~$w$.   We will study
$\BB^{-1}(U)$ in order to understand $\BB^{-1}(w)$ and
$\widetilde{s}|_{\BB^{-1}(w)}$.

Each of the centers $B_i$ that are blown up in the construction of $Y$ are obtained as preimages of the centers $A_i$
under $\widetilde s_0$.  In particular, the only centers that will be blown up to construct $\BB^{-1}(U)$ are the preimages of the centers subordinate
to the primitive centers $\Pi_{S^0},\ldots,\Pi_{S^a}$. 

Each of the points $q_1,\ldots,q_{N+1}$ is totally invariant under $\widetilde s_0$ so that there are no additional preimages of them.  Meanwhile, $q_{N+2}$ 
has $2^N$ preimages, consisting of all points of the form $[1:\pm 1: \pm 1:\cdots: \pm 1]$.  Each of the centers from $B_i$ is the span of $i+1$ of these $N+1+2^N$ points.

Each primitive center $\Pi_{S^0},\ldots,\Pi_{S^a}$  has a unique
preimage under $\widetilde s_0$ that contains the point $w$ (as can be explicitly verified). Let $\Lambda^0,\ldots,\Lambda^a$ be the unique preimages of the primitive
centers that contain $w$.  Each of the further centers that is blown up will be
subordinate to one of these primitive centers and those subordinate to distinct
primitive centers intersect transversally.  In particular, we can  blow up to form $\BB^{-1}(U)$ in precisely the same order as we did to
form $\AA^{-1}(V)$.

%  This is clear for $\Omega^0$ since it
%is totally invariant under $\widetilde s_0$.  For any $1 \leq b \leq a$, notice that any
%preimage of $\Omega_b$ is the span of the points from {\color{red}super or sub?} $S^b \setmius
%\{{N+1}\}$ {\color{red} is this minus $N+1$ or minus $q_{N+1}$?}together with exactly one of the inverse images $[1:\pm 1: \pm
%1:\cdots: \pm 1]$ of $q_{N+1}$.  However, the signs of the entries of this preimage
%of $q_{N+1}$ do not matter for components $c \in S^b \setminus \{{N+1}\}$  {\color{red} is this minus $N+1$ or minus $q_{N+1}$?} {\color{blue}since they can
%be canceled out by adjusting the coefficient of $q_c$.  Meanwhile, for $c \not
%\in S^b \setminus \{{N+1}\}$ the signs of $[1:\pm 1: \pm 1:\cdots: \pm 1]$
%are uniquely determined by the relative signs of the corresponding components
%of $w$.}

Let us first blow up the primitive centers, replacing $U$ by the subset of $U \times \P^{l-1}
 \times \P^{j_1-2} \times \cdots \times \P^{j_a-2}$ given in the coordinates
\begin{eqnarray*}
\{[u_1:u_2:\cdots:u_{N+1}] \times [n^0_1:\cdots:n^0_l] \times [n^1_1:\cdots:n^1_{j_1-1}] \times \cdots \times  [n^a_1:\cdots:n^a_{j_a-1}]
\end{eqnarray*}
by
{\footnotesize
\begin{eqnarray*}
(n^0_1,\cdots,n^0_l) &\sim& (u_1,\ldots,u_l), \\
(n^1_1,\ldots,n^1_{j_1-1}) &\sim& (u_{l+2} \pm u_{l+1},u_{l+3} \mp u_{l+1},\ldots,u_{l+j_1} \mp u_{l+1})  \\
&\vdots& \hspace{1in} \\
(n^j_1,\ldots,n^j_{j_a-1}) &\sim& (u_{l+j_1+\cdots+j_{a-1}+2} \mp u_{l+j_1+\cdots+j_{a-1}+1},\ldots,u_{l+j_1+\cdots+j_a} \mp u_{l+j_1+\cdots+j_{a-1}+1}) \},
\end{eqnarray*} }
%where the signs of the $\mp$ are $-$ if the corresponding components of $w$ are equal and $+$ if the corresponding components of $w$ are opposite.
Let us denote the blow up of $U$ along all of the primitive centers $\Lambda^0,\ldots,\Lambda^a$ by $\mu: U^\# \rightarrow U$.
In particular, the fiber over $w$ is $\mu^{-1}(w) \cong \P^{l-1} \times \P^{j_1-2} \times \cdots \times \P^{j_a-2}$.

Notice that $\widetilde s_0: U \rightarrow V$ is given by $$[v_1:\cdots:v_{N+1}] =
\widetilde s_0([u_1:\cdots:u_{N+1}]) = [u_1^2:\cdots:u_{N+1}^2].$$  By Corollary
\ref{COR:BLOW_UP_PULL_BACK}, this lifts to a holomorphic mapping $s^\# :
U^\# \rightarrow V^\#$ whose restriction $s^\# |_{\mu^{-1}(w)}:
\mu^{-1}(w) 
\rightarrow \nu^{-1}(z)$ is given by
{\footnotesize
\begin{align*}%\label{EQN:LIFT_TO_INTERMEDIATE_FIBER}
s^\# |_{\mu^{-1}(w)} & \left([n^0_1:\cdots:n^0_l],[n^1_1:\cdots:n^1_{j_1-1}],\ldots,[n^k_1:\cdots:n^j_{j_a-1}]\right)  \\ & =  \left([(n^0_1)^2:\cdots:(n^0_l)^2],[n^1_1:\cdots:n^1_{j_1-1}],\ldots,[n^k_1:\cdots:n^j_{j_a-1}]\right).
\end{align*}
}
In other words, the restriction $s^\# |_{\mu^{-1}(w)}: \mu^{-1}(w)\rightarrow \nu^{-1}(z)$ is the squaring map on the first factor and the identity
on each of the remaining factors.

We now blow up all of the centers that are subordinate to $\Lambda^0$.  They are preimages under $\widetilde s_0$ of the centers subordinate to $\Pi_{S^0}$.
In particular, the places where their proper transforms intersect $\mu^{-1}(w)$ are obtained as
the preimages under $s^\#$ of the places where the centers subordinate to $\Pi_{S^0}$ intersect $\nu^{-1}(z)$.  Thus, for all $0 \leq i \leq l-3$, we have 
\begin{eqnarray*}
B_i^{l-1} \times \P^{j_1-2} \times \cdots \times \P^{j_a-2} = (s^\#)^{-1}(A_i^{l-1} \times \P^{j_1-2}\times \cdots \times \P^{j_a-2}).
\end{eqnarray*}
Blowing these centers up, in order of dimension modifies $\mu^{-1}(w)$ to become
\begin{eqnarray*}
Y^{l-1} \times \P^{j_1-2} \times \cdots \times \P^{j_a-2}
\end{eqnarray*}
and the map $s^\#$ lifts to a holomorphic map 
\begin{eqnarray*}
\widetilde{s}^\# : Y^{l-1} \times \P^{j_1-2} \times \cdots \times \P^{j_a-2} \rightarrow X^{l-1} \times \P^{j_1-2} \times \cdots \times \P^{j_a-2}
\end{eqnarray*}
whose action on the first term in the Cartesian product is $\widetilde{s}^{\,l-1}: Y^{l-1}
\rightarrow X^{l-1}$ (by the uniqueness in Corollary \ref{COR:BLOW_UP_PULL_BACK}).
The action on each of the remaining terms of the product is the identity.

We now blow up the centers that are subordinate to $\Lambda^1$.  They are preimages under $\widetilde s_0$ of the centers subordinate to $\Pi_{S^1}$.
In particular, the places where their proper transforms intersect the fiber over $w$ are obtained as
the preimages under $\widetilde{s}^\#$ of the places where the centers subordinate to $\Pi_{S^1}$ intersect the fiber over $z$.  Thus, for all $0 \leq i \leq j_1-3$, we have 
\begin{eqnarray*}
Y^{l-1} \times A_i^{j_1-2} \times \cdots \times \P^{j_a-2} = (\widetilde{s}^\#)^{-1}(X^{l-1} \times A_i^{j_1-2}  \times \cdots \times \P^{j_a-2}).
\end{eqnarray*}
Blowing these centers up in order of dimension modifies the fiber over $w$ to become
\begin{eqnarray*}
Y^{l-1} \times X^{j_1-2} \times  \P^{j_2-2} \times \cdots \times \P^{j_a-2}
\end{eqnarray*}
and the map $\widetilde{s}^\#$ lifts to a holomorphic map
\begin{eqnarray*}
\widehat{s}^\# : Y^{l-1} \times X^{j_1-2} \times \P^{j_2-2} \times \cdots \times \P^{j_a-2} \rightarrow X^{l-1}\times X^{j_1-2} \times \P^{j_2-2} \times \cdots \times \P^{j_a-2}
\end{eqnarray*}
whose action on first term in the Cartesian product remains as $\widetilde{s}^{\, l-1}: Y^{l-1} \rightarrow X^{l-1}$ and whose action on each of the remaining terms 
is the identity.

Continuing this way through each of the factors in the Cartesian product, we conclude that 
(\ref{EQN:WFIBER}) and (\ref{EQN:RESTRICTION}) hold.
In particular, $\widetilde{s}|_{\BB^{-1}(w)}$ has finite fibers.

\vspace{0.1in}

\noindent
We ultimately conclude that $\widetilde{s} : Y \rightarrow X$ has finite fibers.
\qed(Lemma \ref{finitefibers})

\qed(Proposition \ref{goodresolution})

%% file: dynamical_degrees.tex
\section{Computing Dynamical Degrees} \label{SEC:DYN_DEG}

We now focus our attention on computing the dynamical degrees $\lambda_k(f_\rho)$, $1 \leq k \leq N$, for all of the generators $f_\rho$ of the semi-group $\mathcal{F}^N$.
The following three facts simplify our task.
\begin{enumerate}
\item\label{first} Corollary \ref{algstabthm} establishes that for all $n\geq 3$, for all $\rho\in S_n$, the map $f_\rho:X^N\dashrightarrow X^N$ is algebraically stable. As a consequence, 
\[
\lambda_k(f_\rho)=\text{ the spectral radius of }(f_\rho)^\ast:H^{k,k}(X^N;\C)\to H^{k,k}(X^N;\C).
\]
\item\label{second} By Proposition \ref{PROP:FUNCTORIAL}, since $\widetilde s:Y^N\to X^N$ has finite fibers, $f^*=(g\circ s)^*=s^*\circ g^*$ on all $H^{k,k}(X^N;\C)$. 
\item\label{third} Keel's theorem \ref{keelthm} presents the cohomology ring $H^*(X^N;\C)$ as  quotient of the ring generated by all boundary strata by combinatorial relations. \end{enumerate}
Point (\ref{first}) reduces the computation of $\lambda_k(f_\rho)$ to the nondynamical problem of computing 
\[
(f_\rho)^*:H^{k,k}(X^N;\C)\to H^{k,k}(X^N;\C).
\]
 Point (\ref{second}) replaces the computation of $(f_\rho)^*$ by the
computation of \\ $(g_\rho)^*:H^{k,k}(X^N;\C)\to H^{k,k}(X^N;\C)$ and
$s^*:H^{k,k}(X^N;\C)\to H^{k,k}(X^N;\C)$.  This factorization splits the
computation into two natural parts: the combinatorial difficulties arising from
the permutation $\rho\in S_n$ are confined to the automorphism $g_\rho~:~X^N~\to~X^N$, and the difficulties arising from indeterminacy of $f_\rho$ are confined
to a {\em single map} $s:X^N\dashrightarrow X^N$. 

One major complication is that the number of boundary strata, and the dimensions
of the cohomology groups both grow quickly with $N$ as displayed in Table \ref{TABLE_STRATA_DIMS} (above) and Table \ref{TABLE_COH_DIMS} (in Section \ref{SEC:INTRO}).   

\begin{table}
\begin{center}
\begin{tabular}{|c|ccccccc|}
\hline
$N \setminus k$ & 0 & 1 & 2 & 3 & 4 & 5 & 6  \\
0   &      1  &   & & & & & \\
1   &    1   &    1   &    & & & & \\
2   &    1   &    5   &    1   & & & & \\
3   &    1  &    25   &   105 &     1  & & & \\4  &    1  &    56   &   490 &    1260  &    1  & & \\
5   &   1   &   119   & 1918  &   9450  & 17325  &    1  & \\
6   &    1   &    246   &  6825  &  63193 &  197774 &  310677 &  1  \\
\hline
\end{tabular}
\end{center}
\vspace{0.1in}
\caption{\label{TABLE_STRATA_DIMS}Number of strata of codimension $k$ in $X^N$.}
\end{table}

\subsection{Pullback under the automorphism $g_\rho:X^N\to X^N$}
\begin{proposition}\label{PROP:PULL_BACK_AUTO}
For $[D^{S_1} \cap \cdots \cap D^{S_k}]\in H^{k,k}(X^N;\C)$, 
\[
(g_\rho)^*\left([D^{S_1} \cap \cdots \cap  D^{S_k}]\right)=[D^{\rho^{-1}(S_1)}\cap \cdots \cap D^{\rho^{-1}(S_k)}].
\]
\end{proposition}

\begin{proof}
It follows from Proposition \ref{autos} and the fact that $g_\rho$ is unramified that
$(g_\rho)^* (D^S) = D^{\rho^{-1}(S)}$, as divisors.  Thus, on the level of cohomology classes
we have $[g_\rho^* D^S] = [D^{\rho^{-1}( S)}]$.

It then follows easily for the codimension $k$ stratum by Corollary \ref{completeintersection} and taking cup products:
\begin{eqnarray*}
g_\rho^*\left([D^{S_1} \cap \cdots \cap  D^{S_k}]\right) = (g_\rho)^*([D^{S_1}] \cupprod \cdots \cupprod [D^{S_k}])
= g_\rho^*([D^{S_1}]) \cupprod \cdots \cupprod g_\rho^*([D^{S_k}]) \\ = [D^{\rho^{-1}(S_1)}] \cupprod \cdots \cupprod [D^{\rho^{-1}(S_k)}] = [D^{\rho^{-1}(S_1)} \cap \cdots \cap D^{\rho^{-1}(S_k)}].
\end{eqnarray*}
Note that we are using that $g_\rho$ is continuous so that it preserves cup products.
\end{proof}
\medbreak

We will construct an explicit basis ${\bm
B}_k^N$ of $H^{k,k}(X^N;\C)$ consisting of fundamental cohomology classes of certain
codimension $k$ boundary strata.  By Proposition \ref{PROP:PULL_BACK_AUTO},
$g_\rho^*$ induces a permutation on the set of all codimension $k$ boundary
strata and Keel's Theorem \ref{keelthm} can be used to express $g_\rho^*({\bm B}_k^N)$ in terms of ${\bm B}_k^N$.  The stratified structure of $X^N$ coupled with the resulting beautifully simple combinatorics of Keel's Theorem make it possible to directly implement these computations on the computer for all $\rho \in S_n$.  (Our computations were done in Sage \cite{SAGE}.)

\subsection{Pullback action on $H^{1,1}(X^N;\C)$ under the rational map $s:X^N\dashrightarrow X^N$}\label{SEC:PULLVBACK_SQ1}
It will be very helpful for us that $H^{1,1}(X^N;\C)$ is spanned by the
fundamental cohomology classes of the boundary divisors $D^S$.  This follows
immediately from Keel's Theorem \ref{keelthm}.  However, in order to construct
an explicit basis, we recall that for any iterated blow up
$Z$ of projective space, a basis of $H^{1,1}(Z;\C)$ is the fundamental
cohomology class of the proper transform of any hyperplane, together with the
fundamental cohomology classes of each of the exceptional divisors \cite[p. 605]{GH}.

Each of the centers of blow up used in the construction of $X^N$ is (the proper transform of) a linear space of the form
\begin{eqnarray*}
0 = z_{i_1} = \cdots = z_{i_j} \qquad \mbox{or} \qquad z_{i_1} = \cdots = z_{i_j}.
\end{eqnarray*}
In the isomorphism given by Kapranov's Theorem (Theorem \ref{kapranovthm}), the exceptional divisors over these centers correspond to the boundary divisors $D^S$, where
\begin{eqnarray*}
S = \{p_1,p_{i_1 + 2},\ldots,p_{i_j + 2}\}  \qquad \mbox{or} \qquad S =\{p_{i_1 + 2},\ldots,p_{i_j + 2}\},
\end{eqnarray*}
respectively.
Thus, one can take the following as an ordered basis for $H^{1,1}(X^N;\C)$ 
\begin{eqnarray*}
{\bm B}_1^N = \{[D^{S_1}],\ldots,[D^{S_\ell}]\}
\end{eqnarray*}
where $S_1 = \{p_1,p_3\}$ corresponds to the proper transform of the hyperplane $z_1 = 0$ and $S_2,\ldots,S_{\ell}$
are all subsets of $P$ with $2 < |S_i| \leq n-2$ and $p_2 \not \in S_i$.
In particular, $\ell = 2^{n-1} - \binom{n}{2} - 1$.
We order the $S_i$ so that the $2^{n-2}-n+1$ containing $p_1$ are listed before those not containing $p_1$. 

Let us begin by pulling back $[D^S]$ under $\widetilde{s}^*$ for any boundary
divisor $D^S$, independent of whether it appears in ${\bm B}_1^N$.  
If $|S \cap \{p_1,p_2\}| = 1$, then by replacing $S$ with $S^C$, if necessary, we have
$S = \{p_1,p_{i_1},\ldots,p_{i_j}\}$ with $i_q \geq 3$ for $1 \leq q \leq j$. Let $\mathscr{D}^{S}$ denote the divisor in
$Y^N$ obtained as the proper transform of the exceptional divisor obtained by
blowing up (the proper transform of) $0 = z_{i_1} = \cdots = z_{i_j}$.

\begin{lemma}\label{LEM_PULLBACK_RAMIFIED_DIVISOR1}
If $|S \cap \{p_1,p_2\}| = 1$,
then $\widetilde{s}^*\left([D^{S}]\right) = 2 [\mathscr{D}^{S}]$.
\end{lemma}

\begin{proof}
Notice that the whole construction of $\widetilde{s}: Y^N \rightarrow X^N$ from $s_0: \P^N \rightarrow \P^N$
that is outlined in Diagram (\ref{EQN:DIAGRAM_FOR_STILDE})
commutes with any permutation of the underlying homogeneous coordinates of $\P^N$.
Therefore, without loss of generality, we can suppose that $S = \{1,3,4,\ldots,j+2\}$ with
$D^S$ corresponding to the proper transform of $0 = z_1 = \cdots = z_j$.

%Up to this point of the paper, we have used the term ``divisor'' informally, to
%mean codimension one analytic set.  Within this proof and that of Lemma
%\ref{LEM_PULLBACK_UNRAMIFIED_DIVISOR}, we will use the term correctly, meaning
%``locally principal divisor''.  

%To emphasize this distinction, we will continue
%to use $D^S$ to denote the codimension one subset of $X^N$ that have been
%calling a boundary divisor and $(D^S)$ the locally principal divisor that
%vanishes with multiplicity one along this set.  
We will use the notation $(D^S)$ when we consider $D^S$ as a locally principal divisor with multiplicity. It will be somewhat easier to pull back the divisor $(D^S)$ instead of pulling back the cohomology class $[D^S]$.  This 
will be sufficient for our purposes, because of the following commutative diagram, which is adapted to our setting from \cite[p. 139]{GH}:
\begin{eqnarray}\label{CHERN}
\xymatrix{
H^1(X^N,\mathcal{O}^*) \ar[r]^{\widetilde{s}^*} \ar[d]^c & H^1(Y^N,\mathcal{O}^*) \ar[d]^c \\
H^{1,1}(X^N;\C) \ar[r]^{\widetilde{s}^*} & H^{1,1}(Y^N;\C) 
}
\end{eqnarray}
The cohomology groups in the first row describe the linear equivalence classes of locally principal divisors and the vertical arrows
denote the Chern class.   

Throughout our calculations, we will appeal to $s_0: \P^N \rightarrow \P^N$ which is given by 
\begin{eqnarray*}
[w_1:\cdots:w_{N+1}] = [z_1^2: \cdots : z_{N+1}^2] = s_0([z_1:\cdots:z_{N+1}]).
\end{eqnarray*}

The case $j=1$ is special since $D^S \equiv D^{\{p_1,p_3\}}$ corresponds to the
proper transform of $w_1 = 0$ under all of the blow ups used to construct $X^N$
and $\mathscr{D}^S$ corresponds to the proper transform of $z_1 = 0$ under all
of the blow ups used to construct $Y^N$.  Moreover, it's clear from the commutative diagram
(\ref{EQN:DIAGRAM_FOR_STILDE}) that $\widetilde{s}^{-1}(D^S) = \mathscr{D}^S$.  It remains to keep track of multiplicities.
The affine coordinates
$v_1 = \frac{w_1}{w_{N+1}},\ldots,v_N = \frac{w_N}{w_{N+1}}$ serve as local
coordinates on $X^N$ in a neighborhood of generic points of $D^S$  and the
affine coordinates $u_1 = \frac{z_1}{z_{N+1}},\ldots,u_N = \frac{z_N}{z_{N+1}}$
serve as local coordinates on $Y^N$ in a neighborhood of generic points of
$\mathscr{D}^S$.   (Here, ``generic'' means points which are not on any of the exceptional divisors). Since $(D^S)$ is given locally at generic 
points by $v_1 = 0$ and $\widetilde{s}(u_1,\ldots,u_N) = (u_1^2,\ldots,u_N^2)$ we have that $\widetilde{s}^*((D^S))$
is given locally at generic points by $u_1^2 = 0$.  This gives $\widetilde{s}^*\left((D^{S})\right) = 2 (\mathscr{D}^{S})$ and hence $\widetilde{s}^*\left([D^{S}]\right) = 2 [\mathscr{D}^{S}]$, by Diagram (\ref{CHERN}).

The case $j > 1$ will be similar, except that we need to describe generic
points of $D^S$ and $\mathscr{D}^S$ using blow up coordinates.  Let us again use the affine coordinates $(v_1,\ldots,v_N)$ on $\P^N \setminus \{w_{N+1} = 0\}$ and $(u_1,\ldots,u_N)$ on $\P^N \setminus \{z_{N+1} = 0\}$. At points of the 
proper transform of $0 = v_1 = \ldots = v_j$ not lying on any exceptional divisors
resulting from blow ups of lower dimensional centers, the blow up of this center is given by
\begin{eqnarray*}
\{(v_1,\ldots,v_N)\times [n_1:\cdots:n_{j}] \in \C^N \times \P^{j-1} \, | \, (v_1,\ldots,v_{j}) \sim (n_1,\ldots,n_{j})\}.
\end{eqnarray*}
Local coordinates on $X^N$ in a neighborhood of generic points of $D^S$ are given by $\left(\frac{n_1}{n_j},\ldots, \frac{n_{j-1}}{n_{j}},v_{j},\ldots,v_N\right)$ and in these coordinates $(D^S)$ is given by $v_{j} = 0$.

Generic points of $\mathscr{D}^S$ can be described by the blow up of $0 = u_1 = \cdots = u_j$, which is given by
\begin{eqnarray*}
\{(u_1,\ldots,u_N)\times[m_1:\cdots:m_{j}] \in \C^N \times \P^{j-1} \, | \, (u_1,\ldots,u_{j}) \sim (m_1,\ldots,m_{j})\}.
\end{eqnarray*}
Similarly, local coordinates on $Y^N$ in a neighborhood of generic points of
$\mathscr{D}^S$ are given by $\left(\frac{m_1}{m_j},\ldots, \frac{m_{j-1}}{m_{j}},u_{j},\ldots,u_N\right)$.   In these
systems of local coordinates, we have
{\small
\begin{eqnarray*}
\widetilde{s}  \left(\frac{m_1}{m_j},\ldots, \frac{m_{j-1}}{m_{j}},u_{j},\ldots,u_N\right) 
  = \left(\left(\frac{m_1}{m_{j}}\right)^2,\ldots, \left(\frac{m_{j-1}}{m_{j}}\right)^2,u_{j}^2,\ldots,u_N^2\right) \\
\end{eqnarray*}
}
Therefore, at generic points of $\mathscr{D}^S$, $\widetilde{s}^*((D^S))$ is given by $u_{j}^2 = 0$.  
This gives $\widetilde{s}^*\left((D^{S})\right)~=~2 (\mathscr{D}^{S})$ and hence $\widetilde{s}^*\left([D^{S}]\right) = 2 [\mathscr{D}^{S}]$, by Diagram (\ref{CHERN}).
\end{proof}

If $|S \cap \{p_1,p_2\}| = 0$ or $2$, then replacing $S$ with $S^C$, if necessary, we have
$S = \{p_{i_1},\ldots,p_{i_j}\}$ with $i_q \neq 1,2$ for $1 \leq q \leq
j$.  Let $\mathscr{D}^{S}_{\pm \pm \ldots \pm}$ denote the divisor in
$Y^N$ obtained as proper transform of the exceptional divisor obtained by
blowing up (the proper transform of) 
\begin{eqnarray}\label{PM_EQN}
z_{i_1} = \pm z_{i_2} = \cdots = \pm z_{i_j}.
\end{eqnarray}

\begin{lemma}\label{LEM_PULLBACK_UNRAMIFIED_DIVISOR}
If $|S \cap \{p_1,p_2\}| = 0$ or $2$, then
\begin{eqnarray}\label{EQN:UNRAMIFIED_PULLBACK}
\widetilde{s}^*\left([D^{S}]\right) = \sum [\mathscr{D}^{S}_{\pm \pm \ldots \pm}],
\end{eqnarray}
where the sum is taken over the $2^{j-1}$ possible choices of signs in (\ref{PM_EQN}).
\end{lemma}

\begin{proof}
The proof will be quite similar to that of Lemma \ref{LEM_PULLBACK_RAMIFIED_DIVISOR1}.
It will again be 
simpler to pull back the divisor $(D^S)$ rather than the cohomology class $[D^S]$ and we can again
assume, without loss of generality, that $S = \{3,4,\ldots,2+j\}$. 

If $j=2$, $D^{S} \equiv D^{\{p_3,p_4\}}$ is the proper transform of $z_1 = z_2$
under all of the blow ups used to construct $X^N$ from $\P^N$.  Similarly,
$\mathscr{D}^S_\pm$ is the proper transform of $z_1 = \pm z_2$ under all of the
blow ups used to construct $Y^N$.  The local coordinates $v_1,\ldots,v_N$ on
$\P^N \setminus \{w_{N+1} = 0\}$ from Lemma
\ref{LEM_PULLBACK_RAMIFIED_DIVISOR1} serve as local coordinates on $X^N$ in a
neighborhood of generic points of $D^S$.  Meanwhile, the local coordinates
$u_1,\ldots,u_N$ in $\P^N \setminus \{z_{N+1} = 0\}$ serve as local coordinates
on $Y^N$ in a neighborhood of generic points in a neighborhood of
$\mathscr{D}^S_\pm$.  Moreover, it's clear from the commutative
diagram (\ref{EQN:DIAGRAM_FOR_STILDE}) that $\widetilde{s}^{-1}(D^S) =
\mathscr{D}^S_+ \cup \mathscr{D}^S_-$.  Thus, it remains to keep track of
multiplicities.  The divisor $(D^S)$ is locally given at generic points by $v_1
- v_2 = 0$.   At points of $Y^N$ where $u_1,\ldots,u_N$ serve as coordinates
  and at points of $X^N$ where $v_1,\ldots,v_N$ serve as coordinates, the map
$\widetilde{s}$ is given by $\widetilde{s}(u_1,\ldots,u_N) = (u_1^2,\ldots,u_N^2)$.
Since $(D^S)$ is locally given by $v_1-v_2 = 0$, it follows that
$\widetilde{s}^*((D^S))$ is given at generic points by $u_1^2 - u_2^2 =
(u_1-u_2)(u_1+u_2)$.  Since the first factor describes $(\mathscr{D}^S_+)$ and
the second factor describes $(\mathscr{D}^S_-)$, we conclude that
$\widetilde{s}^*((D^S)) = (\mathscr{D}^S_+) + (\mathscr{D}^S_-)$.  By commutative
diagram (\ref{CHERN}), this implies $\widetilde{s}^*([D^S]) = [\mathscr{D}^S_+] +
[\mathscr{D}^S_-]$.

The case $j > 2$ will be similar, except that we will need to use blow up coordinates.  
It's clear from the commutative
diagram (\ref{EQN:DIAGRAM_FOR_STILDE}) that 
\begin{eqnarray*}
\widetilde{s}^{-1}(D^S) = \bigcup \mathscr{D}^S_{\pm\cdots\pm}.
\end{eqnarray*}
Thus, it remains to compute the multiplicity of each contribution.
Let us again use the affine coordinates 
$(v_1,\ldots,v_N)$ on $\P^N \setminus \{w_{N+1} = 0\}$ and $(u_1,\ldots,u_N)$ on $\P^N \setminus \{z_{N+
1} = 0\}$. At points of the
proper transform of $v_1 = \ldots = v_j$ not lying on any exceptional divisors
resulting from blow ups of lower dimensional centers, the blow up of this center is given by
\begin{eqnarray*}
\{(v_1,\ldots,v_N)\times [n_1:\cdots:n_{j-1}] \in \C^N \times \P^{j-2} \, | \, (v_1-v_{j},\ldots,v_{j-1}-v_{j}) \sim (n_1,\ldots ,n_{j-1})\}.
\end{eqnarray*}
Local coordinates on $X^N$ in a neighborhood of generic points of $D^S$ are
given by
\begin{eqnarray*}
\left(\frac{n_1}{n_{j-1}},\ldots,\frac{n_{j-2}}{n_{j-1}},v_{j-1}-v_j,v_j,\ldots,v_N\right).
\end{eqnarray*}
Generic points of $\mathscr{D}^S_{\pm \pm \cdots \pm}$ can be described by the blow up of $z_1 = \pm z_2 = \pm z_j$, which is given by 
\begin{eqnarray*}
\{(u_1,\ldots,u_N)\times [m_1:\cdots:m_{j-1}] \in \C^N \times \P^{j-2} \, | \, (v_1\mp v_j,\ldots,v_{j-1} \mp v_j) \sim (m_1,\ldots ,m_{j-1})\}.
\end{eqnarray*} 
Local coordinates on $Y^N$ in a neighborhood of generic points of $\mathscr{D}^S_{\pm \pm \cdots \pm}$ are
given by
\begin{eqnarray*}
\left(\frac{m_1}{m_{j-1}},\ldots,\frac{m_{j-2}}{m_{j-1}},u_{j-1} \mp u_j,u_j,\ldots,u_N\right).
\end{eqnarray*}
In these local coordinates,
{\small
\begin{eqnarray*}
\widetilde{s} \left(\frac{m_1}{m_{j-1}},\ldots,\frac{m_{j-2}}{m_{j-1}},u_{j-1} \mp u_j,u_j,\ldots,u_N\right) 
= \left(\frac{m_1}{m_{j-1}},\ldots,\frac{m_{j-2}}{m_{j-1}},u_{j-1}^2 \mp u^2_j,u^2_j,\ldots,u^2_N\right).
\end{eqnarray*}
}

Since $(D^S)$ is given by $v_{j-1} - v_j = 0$, in these coordinates $\widetilde{s}^*((D^S))$ is given
by $u_{j-1}^2 \mp u^2_j = (u_{j-1} - u_j)(u_{j-1}+u_j)$.  Since $(\mathscr{D}^S_{\pm \pm \cdots \pm})$
is given locally by exactly one of these two linear factors, we see that for each combination of $\pm$, the preimage
$\mathscr{D}^S_{\pm \pm \cdots \pm}$ is counted with multiplicity one.  Thus, we have
\begin{eqnarray}
\widetilde{s}^*\left((D^{S_i})\right) = \sum (\mathscr{D}^{S_i}_{\pm \pm \ldots \pm}),
\end{eqnarray}
By commutative diagram (\ref{CHERN}) this gives (\ref{EQN:UNRAMIFIED_PULLBACK}).
\end{proof}

\begin{remark}
We will refer to divisors $D^S$ with $|S \cap \{p_1,p_2\}| = 1$ as {\em ramified divisors}
and those with $|S \cap \{p_1,p_2\}| = 0$ or $2$ as {\em unramified divisors}.
\end{remark}

We now return to our basis 
\begin{eqnarray*}{\bm B}_1^N = \{[D^{S_1}],\ldots,[D^{S_\ell}]\}\end{eqnarray*}
where $S_1 = \{p_1,p_3\}$ corresponds to the proper transform of the hyperplane $z_1 = 0$ and $S_2,\ldots,S_{\ell}$
are all subsets of $P$ with $2 < |S_i| \leq n-2$ and $p_2 \not \in S_i$.

\begin{proposition}\label{PROP:H11_PULLBACK}
With respect to the ordered basis ${\bm B}_1^N$
\begin{eqnarray*}
s^*: H^{1,1}(X^N;\C) \rightarrow H^{1,1}(X^N;\C) \quad \mbox{is given by} \quad
s^* = {\rm diag}(2,\ldots,2,1,\ldots,1),
\end{eqnarray*}
where the first $2^{n-2}-n+1$ entries of the diagonal are $2$, corresponding to
the ramified divisors $D^S$ (those with $|S \cap \{p_1,p_2\}| = 1$), and the remaining
entries are $1$, corresponding to the unramified divisors $D^S$ (those with $|S \cap
\{p_1,p_2\}| = 0$).
\end{proposition}

\begin{proof}
For any $[D^S_i]$ we compute $s^*([D^S_i]) = \pr_*(\widetilde{s}^*([D^S_i]))$.  We
will us Lemmas \ref{LEM_PULLBACK_RAMIFIED_DIVISOR1} and
\ref{LEM_PULLBACK_UNRAMIFIED_DIVISOR} to compute $\widetilde{s}^*([D^S_i])$.  We
will then use Lemma \ref{LEM:PUSH_FORWARD_VARIETIES} to determine the affect of
$\pr_*$ on each of the fundamental classes in $\widetilde{s}^*([D^S_i])$.  

Since $\pr: Y^N \rightarrow X^N$ is a birational morphism, it follows from
Zariski's Main Theorem \cite[Ch. III, Cor. 11.4]{HART} that the fibers of $\pr$ are connected.  In
particular, for any irreducible subvariety $V \subseteq Y^N$ we will either have
$\dim(\pr(V)) < \dim(V)$ or $\deg_{\rm top}(\pr|_V) = 1$.  

First, suppose that $S_i = \{p_1,p_{i_1},\ldots,p_{i_j}\}$ with $i_q \neq 2$
for $1 \leq q \leq j$.  According to Lemma \ref{LEM_PULLBACK_RAMIFIED_DIVISOR1}
we have $\widetilde{s}^*([D^{S_i}]) = 2[\mathscr{D}^{S_i}]$.  The homogeneous
coordinates $[z_1:\cdots:z_{N+1}]$ serve as coordinates on generic points of
$Y^N$ and the homogeneous coordinates $[w_1:\cdots:w_{N+1}]$ serve as
coordinates on generic points of $X^N$.  It follows from commutativity of
(\ref{EQN:LADDER1}) that in these coordinates $\pr([z_1:\cdots:z_{N+1}]) =
[z_1:\cdots:z_{N+1}]$.  Since the proper transform of $0 = z_{i_i+2} =
z_{i_j+2}$ is blown up to construct $Y^N$, corresponding to $\mathscr{D}^{S_i}$,
and the proper transform of $0 = w_{i_i+2} = w_{i_j+2}$ is blown up in the
construction of $X^N$, corresponding to $D^{S_i}$, it follows that $\pr$ maps
$\mathscr{D}^{S_i}$ onto $D^{S_i}$.  Therefore,
\begin{eqnarray*}
s^*([D^{S_i}]) = \pr_*(\widetilde{s}^*([D^{S_i}])) = \pr_*(2[\mathscr{D}^{S_i}]) = 2[D^{S_i}].
\end{eqnarray*}

Now, suppose that $S_i = \{p_{i_1},\ldots,p_{i_j}\}$ with $i_q \neq 1,2$
for $1 \leq q \leq j$.  According to Lemma \ref{LEM_PULLBACK_UNRAMIFIED_DIVISOR}
we have 
\begin{eqnarray*}
\widetilde{s}^*([D^{S_i}]) =  \sum [\mathscr{D}^{S_i}_{\pm\ldots \pm}].
\end{eqnarray*}
As in the previous paragraph, $\pr(\mathscr{D}^{S_i}_{+\ldots+}) = D^{S_i}$
implying that $\pr_*([\mathscr{D}^{S_i}_{+\ldots+}]) = [D^{S_i}]$.  Now, consider the case that
not all of the signs indexing $\mathscr{D}^{S_i}_{\pm\ldots \pm}$
are `$+$'.  
First, notice that since $\mathscr{D}^{S_i}_{\pm\ldots \pm}$ is irreducible, so is $\pr(\mathscr{D}^{S_i}_{\pm\ldots \pm})$.
By commutativity of (\ref{EQN:LADDER1}), we have that $\pr(\mathscr{D}^{S_i}_{\pm\ldots \pm})$
lies within 
\begin{eqnarray*}
\mathcal{A}^{-1}(z_{i_1-2} = \pm z_{i_2-2} = \pm z_{i_j-2})
\end{eqnarray*}
where $\mathcal{A}: X^N \rightarrow \P^N$ is the composition of all of the blow ups used to construct $X^N$.  Moreover, $\pr(\mathscr{D}^{S_i}_{\pm\ldots \pm})$ contains at least one point in each $\mathcal{A}$-fiber over $z_{i_1-2} = \pm z_{i_2-2} = \pm z_{i_j-2}$.  However, since at least one of the $\pm$ is minus, generic points of this linear subspace aren't on any of the centers of blow up.  Thus, 
there are points of $\pr(\mathscr{D}^{S_i}_{\pm\ldots \pm})$ which have a neighborhood in $\pr(\mathscr{D}^{S_i}_{\pm\ldots \pm})$ that is contained within a dimension $< N-1$ analytic set.  Since $\pr(\mathscr{D}^{S_i}_{\pm\ldots \pm})$ is irreducible, this implies that
$\dim(\pr(\mathscr{D}^{S_i}_{\pm\ldots \pm})) < N-1$.
Therefore, if not all of the signs are $+$, we have $\pr_*([\mathscr{D}^{S_i}_{\pm\ldots \pm}])~=~0$.  We conclude that 
\begin{eqnarray*}
s^*([D^{S_i}]) =  \pr_*(\widetilde{s}^*([D^{S_i}])) = \pr_*\left(\sum [\mathscr{D}^{S_i}_{\pm\ldots \pm}]\right) = \pr_*([\mathscr{D}^{S_i}_{+\ldots+}]) = [D^{S_i}].
\end{eqnarray*}
\end{proof}

\subsection{Pullback action on $H^{2,2}(X^3;\C)$ under the rational map
$s:X^3\dashrightarrow X^3$}\label{SEC:PULLVBACK_SQ2}
 %for $k =2$ and $N=3$}\label{SEC:PULLVBACK_SQ2}
%move lemma 4.1 here and state corollary about the critical locus

For any projective manifold $X$ and any dominant rational map $f: X
\dashrightarrow X$, it can be quite subtle to keep track of inverse images of
subvarieties $V \subseteq X$ of codimension at least $2$, since they may lie in
the indeterminacy locus $I_f$.  For this reason, one must compute preimages (set-theoretic and cohomological) using a resolution of singularities as in (\ref{PULLBACK_DIAGRAM}).

This is even more subtle for the map $s: X^N \dashrightarrow X^N$ because $\pr:
Y^N \rightarrow X^N$ is defined implicitly by a universal property.  Because of
these challenges and the computational complexity arising from the dimensions
of the $H^{k,k}(X^N; \C)$ growing exponentially with $N$, we will limit
ourselves in this section to $N = 3$. The case for $N=3$, and $k=1$ has already
been analyzed in Section \ref{SEC:PULLVBACK_SQ1}, so we will focus on $k=2$.
With enough computational power and careful book-keeping about preimages lying
in the indeterminacy locus, we expect these techniques to extend to arbitrary
$N$ and $k$.

One further challenge is that we don't know if the action can be expressed by a diagonal matrix:
\begin{question}
For any $N$ and any $k \geq 2$ does there exist an ordered basis for $H^{k,k}(X;\C)$ consisting of fundamental classes of boundary strata,
in which the action of $s^*$ is expressed by a diagonal matrix?  (Compare to Proposition \ref{PROP_PULLBACK_DIM3}, below.)
\end{question}

The following proposition is stated for general $k$ and $N$.  

\begin{proposition}\label{ai_prop}
Let $Z := D^{S_1} \cap \cdots \cap D^{S_k}\subseteq X^N$ be a codimension $k$ boundary stratum and let
$W_1,\ldots,W_\ell \subseteq X^N$ be the irreducible components of $s^{-1}(Z) := \pr(\widetilde{s}^{-1}(Z))$
that have codimension exactly $k$.

We have
\begin{eqnarray*}
s^*[Z] = \sum_{m=1}^\ell 2^r \, [W_m]
\end{eqnarray*}
where $r$ is the number of the boundary divisors among  $\{D^{S_1},\ldots,D^{S_k}\}$ that are ramified, i.e. those satisfying $|S_i \cap \{p_1,p_2\}| = 1$.
\end{proposition}

Some comments are in order:
\begin{enumerate}
\item Since $\widetilde{s}$ has finite fibers, every irreducible component of $s^{-1}(Z)$
has codimension at least $k$.  We ignore any preimages of codimension greater than $k$,
even those lying entirely in $I_s\subseteq X^N$.
\item By Lemma \ref{intersectingdivslemma}, $Z$ is uniquely represented as an intersection of boundary divisors, so that the number $r$ is well-defined.
%\item Lemma \ref{CRITLOCUS} gives that $D^S$ is critical iff $|S \cap \{p_1,p_2\}| =1$.
%\item The fact that each $W_m$ receives the same coefficient $2^r$ follows from...
\end{enumerate}

\begin{proof}
Recall that $s^*([Z]):= \pr_*(\widetilde{s}^*([Z]))$.
Without loss of generality, we can suppose that the first $r$ divisors are ramified and the remaining
ones are unramified.
From Lemmas \ref{LEM_PULLBACK_RAMIFIED_DIVISOR1} and \ref{LEM_PULLBACK_UNRAMIFIED_DIVISOR} we have
\begin{eqnarray*}
\widetilde{s}^*([Z]) &=& \widetilde{s}^*([D^{S_1}] \cupprod \cdots \cupprod [D^{S_k}]) =
\widetilde{s}^*([D^{S_1}]) \cupprod \cdots \cupprod \widetilde{s}^*([D^{S_k}])  \\
&=& 2[\mathscr{D}^{S_1}] \cupprod \cdots \cupprod 2[\mathscr{D}^{S_r}] \cupprod \left(\sum [\mathscr{D}^{S_{r+1}}_{\pm \pm \ldots \pm}]\right) \cupprod \cdots \cupprod \left(\sum [\mathscr{D}^{S_{k}}_{\pm \pm \ldots \pm}]\right) \\ &=& 2^r \sum_{m=1}^j [V_m].
\end{eqnarray*}
where each $V_m$ is an irreducible  component of $\widetilde{s}^{-1}(D^{S_1}) \cap
\cdots \cap \widetilde{s}^{-1}(D^{S_k})$.  Each of these components has codimension $k$ since
$\widetilde{s}$ has finite fibers.    Notice that each $k$-fold iterated cup product obtained when expanding the sum corresponds to $k$ fundamental classes of divisors intersecting transversally.  This is why
the fundamental cohomology class of each component $V_m$ does not get an extra multiplicity.

According to Lemma \ref{LEM:PUSH_FORWARD_VARIETIES}, any $V_m$ with
$\dim(\pr(V_m)) < \dim(V_m)$ will have $\pr_*([V_m]) = 0$.  Removing any such
$V_m$ from our list (and re-ordering if necessary), we can assume that the first $\ell$ 
components $V_1,\ldots,V_\ell$ are mapped by $\pr$ onto
$W_1,\ldots,W_\ell$ of the same dimension and the remaining components are decreased in dimension by the map $\pr$.  It follows from Zariski's Main Theorem 
\cite[Ch. III, Cor. 11.4]{HART} that $\deg_{\rm top}(\pr|_{V_m}) = 1$ for $1 \leq m \leq \ell$.
In conclusion
\begin{eqnarray*}
s^*([Z]) = \pr_*(\widetilde{s}^*([Z])) = \pr_*\left(2^r \sum_{m=1}^{j} [V_m]\right) = 2^r \sum_{m=1}^\ell [W_m].
\end{eqnarray*}
\end{proof}

We will construct an ordered basis $\bm B_2^3$ for $H^{2,2}(X^3,\mathbb{C})$ using intersections of the boundary divisors indexed by the following subsets of $P = \{p_1,\ldots,p_6\}$:
\begin{eqnarray*}
\begin{array}{cccc}
S_1:=\{p_1,p_3,p_4,p_5\} &  S_6:=\{p_1,p_3,p_4\}   &   S_{11}:=\{p_1,p_5,p_6\}  &   S_{16}:=\{p_1,p_2,p_6\} \\
S_2:=\{p_1,p_3,p_4,p_6\} &  S_7:=\{p_1,p_3,p_5\}   &   S_{12}:=\{p_1,p_2\}      &   S_{17}:=\{p_1,p_4\} \\
S_3:=\{p_1,p_3,p_5,p_6\} &  S_8:=\{p_1,p_3,p_6\}   &   S_{13}:=\{p_1,p_2,p_3\}  &   S_{18}:=\{p_3,p_4\} \\
S_4:=\{p_1,p_4,p_5,p_6\} &  S_9:=\{p_1,p_4,p_5\}   &   S_{14}:=\{p_1,p_2,p_4\}  &   \\
S_5:=\{p_1,p_3\}	 &  S_{10}:=\{p_1,p_4,p_6\} &  S_{15} := \{p_1,p_2,p_5\} &
\end{array}
\end{eqnarray*}

By Keel's theorem \ref{keelthm}, $\mathrm{dim}(H^{2,2}(X^3;\C))=16$.
We will use as many ramified codimension 2 boundary strata as possible in order to make the expression of $s^*$ in $\bm B_2^3$
as close to being diagonal as possible.
Our first $11$ strata are obtained as intersections of two ramified divisors:
\begin{eqnarray*}
\begin{array}{cccc}
Z_1:=D^{S_1} \cap D^{S_5} & Z_4:=D^{S_2}\cap D^{S_5}  & Z_7:=D^{S_2} \cap D^{S_{10}}  & Z_{10}:=D^{S_4}\cap D^{S_{17}} \\
Z_2:=D^{S_1} \cap D^{S_7} & Z_5:=D^{S_2}\cap D^{S_6}  & Z_{8}:=D^{S_3}\cap D^{S_5}    & Z_{11}:=D^{S_5}\cap D^{S_7}  \\
Z_3:=D^{S_1} \cap D^{S_9} & Z_6:=D^{S_2} \cap D^{S_8} & Z_{9}:=D^{S_3} \cap D^{S_{11}} & \\
\end{array}
\end{eqnarray*}
Let $Z = D^{S_i} \cap D^{S_j}$ be any one of these $11$ boundary strata.
Since each is ramified we have $\widetilde{s}^{-1}(D^{S_i}) = \mathscr{D}^{S_i}$
and $\widetilde{s}^{-1}(D^{S_j}) = \mathscr{D}^{S_j}$.   Hence, $\widetilde{s}^{-1}(Z) = \mathscr{D}^{S_i} \cap \mathscr{D}^{S_j}$.   Similarly to the proof of Proposition \ref{PROP:H11_PULLBACK} we have
$\pr(\mathscr{D}^{S_i} \cap \mathscr{D}^{S_j}) = Z$.
Since $D^{S_i}$ and $D^{S_j}$ both are ramified, it follows
from Proposition \ref{ai_prop} that $s^*([Z]) = 2^2[Z]$.   In summary:
\[
s^*([Z_i]) = 2^2 [Z_i], \text{ for all } 1 \leq i \leq 11.
\]

There are four more ramified strata of codimension $2$ we will use for our basis:
\begin{eqnarray*}
 Z_{12}:=D^{S_1}\cap D^{S_{16}}, \,\,
 Z_{13}:=D^{S_2}\cap D^{S_{15}}, \,\,
 Z_{14}:=D^{S_3}\cap D^{S_{14}}, \,\, \mbox{and} \,\,
 Z_{15}:=D^{S_4}\cap D^{S_{13}}.
 \end{eqnarray*}
For each of them, the first term in the intersection is ramified and the second one is unramified.

First consider $Z_{12} = D^{S_1}\cap D^{S_{16}}$.  
Recall that we use the normalization
\begin{align}\label{6PTS_NORMALIZATION}
\varphi(p_1) = 0, \,\, \varphi(p_2) = \infty, \,\,  \varphi(p_3) = z_1, \,\, \varphi(p_4) = z_2, \,\, \varphi(p_5) = z_3, \,\, \mbox{ and } \,\, \varphi(p_6) = z_4.
\end{align} 
With respect to the coordinates obtained from this normalization $$D^{S_1} =
\mathcal{A}^{-1}([0:0:0:1]) \cong X^2$$ and $\mathscr{D}^{S_1} =
\widetilde{s}^{-1}(D^{S_1}) = \mathcal{B}^{-1}([0:0:0:1]) \cong Y^2$.  
As in the proof of Lemma \ref{finitefibers}, under these identifications of the fibers with $X^2$ and $Y^2$
we have that
\begin{eqnarray*}
\widetilde{s}|_{\mathcal{B}^{-1}([0:0:0:1])} : \mathcal{B}^{-1}([0:0:0:1])
\rightarrow \mathcal{A}^{-1}([0:0:0:1])
\end{eqnarray*}
 is $\widetilde{s}^2: Y^2 \rightarrow X^2$
and
\begin{eqnarray*}
\pr|_\mathcal{B}^{-1}([0:0:0:1]) : \mathcal{B}^{-1}([0:0:0:1])
\rightarrow  \mathcal{A}^{-1}([0:0:0:1])
\end{eqnarray*}
 is $\pr^2: Y^2 \rightarrow X^2$.  
(We are using superscripts on the names of the maps to denote
the dimensions of the domain/codomain, as in Lemma \ref{finitefibers}.)  Therefore, computing $s^{-1}(Z_{12}):=\pr(\widetilde{s}^{-1}(Z_{12}))$
amounts to computing $(s^2)^{-1}(W) := \pr^2((\widetilde{s}^2)^{-1}(W))$, where $W$ is the divisor in $X^2$
obtained from intersecting $\mathcal{A}^{-1}([0:0:0:1]) \cong X^2$ with $D^{S_{16}}$.  
Recall that $X^2$ is the blow up of $\P^2$ at $[0:0:1], [0:1:0], [1:0:0]$ and $[1:1:1]$, where in this context $\P^2$ is the exceptional divisor over $[0:0:0:1]$ obtained in the first round of blow ups of $\P^3$ used to construct $X^3$.
Since $D^{S_{16}}$ is
obtained by blowing up the proper transform of $z_1 = z_2 = z_3$, the
intersection $W$ corresponds in
$\mathcal{A}^{-1}([0:0:0:1])$ to the blow up of $\P^2$ at $[1:1:1]$.  Therefore,
the preimages under $\widetilde{s}^2$ will correspond in
$\mathcal{B}^{-1}([0:0:0:1])$ to the blow ups of $\P^2$ at $[1:\pm 1: \pm 1]$.
As in the proof of Proposition \ref{PROP:H11_PULLBACK}, $\pr^2$ will crush each
of these blow ups other than the one at $[1:1:1]$.  Therefore, the only component of $(s^2)^{-1}(W) := \pr^2((\widetilde{s}^2)^{-1}(W))$ of dimension $1$ is the blow up of $\P^2$ at $[1:1:1]$.  Considered in $X^3$,
this is just $Z_{12}$.

Essentially the same proof shows that for $13 \leq i \leq 15$, the only component of $s^{-1}(Z_i)$ having dimension one is $Z_i$.  Proposition \ref{ai_prop} gives that
 \[
 s^*([Z_i])=2[Z_i],\text{ for all }12\leq i \leq 15.
 \]

We require one more basis element, which unfortunately will not pullback to a multiple of itself.  Let 
\begin{eqnarray*}
Z_{16}:=D^{S_{12}} \cap D^{S_{18}}.
\end{eqnarray*}
It can be readily verified from Theorem \ref{keelthm} that the ordered set 
\[
{\bm B}_2^3 = \{[Z_1],\ldots,[Z_{16}]\}
\]
is a basis of $H^{2,2}(X^3;\C)$. 

We'll again need to use coordinates from the blow up description of $X^3$ given
in Section~\ref{kapsect} to compute $s^*([Z_{16}])$. 
   To simplify notation, let's write $Z \equiv Z_{16}$. 
Since neither $D^{S_{12}}$ or $D^{S_{18}}$ is ramified, Proposition \ref{ai_prop} gives that $s^*[Z]$
will be the sum of fundamental classes of the components of $s^{-1}(Z)$ of dimension $1$, each with multiplicity one.

Recall that we use the normalization stated in Line (\ref{6PTS_NORMALIZATION}).
In these coordinates, $Z$ is the intersection of
$E_{[1:1:1:1]}$ with the proper transform of the hyperplane $z_1 = z_2$.
Consider the preimages $[1:\pm 1:\pm 1:\pm 1] \in s_0^{-1}([1:1:1:1])$.
Because these points are not critical, there are neighborhoods $\BB^{-1}(U)$ and $\AA^{-1}(V)$ of  $\mathcal{B}^{-1}([1:\pm 1:\pm 1:\pm 1])$
and $\mathcal{A}^{-1}([1: 1:1:1])$, respectively, so that
\begin{eqnarray*}
\widetilde{s}|_{\mathcal{B}^{-1}(U)}: \mathcal{B}^{-1}(U) \rightarrow \mathcal{A}^{-1}(V)
\end{eqnarray*}
is a biholomorphism (see the proof of Lemma \ref{finitefibers}).  In particular, there are 
eight irreducible preimages $\mathscr{Z}_{\pm,\pm,\pm}$ of $Z$ under $\widetilde{s}$, with one of them in each such neighborhood $\BB^{-1}(U)$.
We must determine which of them have image under the map $\pr$ of dimension $1$.

There is a sufficiently small neighborhood $U$ of $[1:1:1:1]$ so that $\pr$ maps $\mathcal{B}^{-1}(U)$ biholomorphically
onto $\mathcal{A}^{-1}(U)$.  In particular, $\pr(\mathscr{Z}_{+,+,+}) = Z$, so that $[Z]$ contributes to $s^*([Z])$.
%a center from $B^0 \cup B^1$ of the blow ups used to make $Y^3$ intersects 
%$U$ if and only if it is one of the centers from $A^0 \cup A^1$ that was used
%to make $X^3$.  
%from $B^0 \cup B^1$ of the blow ups used to make $Y^3$ intersects 
%$U$ if and only if it is one of the centers from $A^0 \cup A^1$ that was used
%to make $X^3$.  
%Now, consider the three components indexed by two minus signs.  Since the
%points $[1:\pm 1:\pm 1:\pm 1]$ with exactly two minus signs are not on $A^0
%\cup A^1$, for these points we have $\mathcal{A}^{-1}([1:\pm 1:\pm 1:\pm 1])$
%is a single point.   

Now consider the three components indexed by two minus signs.  Since the points $[1:\pm 1:\pm 1:\pm 1]$
with exactly two minus signs are not on $A_0 \cup A_1$, for these points we have that
$\mathcal{A}^{-1}([1:\pm 1:\pm 1:\pm 1])$ is a single point.
Commutativity of the diagram
\[
\xymatrix{
Y^3\ar[d]^{\mathrm {pr}}\ar[rd]^{\widetilde s}\\
X^3\ar[d]^{\mathcal A}\ar @{-->}[r]^{s} &X^3\ar[d]^{\mathcal A}\\
\P^3\ar[r]^{s_0} &\P^3}
\]
implies 
$\pr(\mathscr{Z}_{\pm,\pm,\pm}) \subseteq \mathcal{A}^{-1}([1:\pm 1:\pm 1:\pm 1])$.
Therefore, the fundamental classes of these preimages do not contribute to $s^*([Z])$.

The remaining four components $\mathscr{Z}_{-1,1,1}, \mathscr{Z}_{1,-1,1}, \mathscr{Z}_{1,1,-1},$ and
$\mathscr{Z}_{-1,-1,-1}$ require a more careful analysis because the corresponding points
in $\P^3$ satisfy 
\begin{eqnarray*}
[1:-1:1:1],[1:1:-1:1],[1:1:1:-1],[1:-1:-1:-1] \in A^1
\end{eqnarray*}
resulting in
\begin{eqnarray*}
\pr|_{\mathcal{B}^{-1}([1:\pm 1: \pm 1: \pm 1])} : \mathcal{B}^{-1}([1:\pm 1: \pm 1: \pm 1]) \rightarrow \mathcal{A}^{-1}([1:\pm 1: \pm 1: \pm 1])
\end{eqnarray*}
being a map from the two-dimensional manifold $\mathcal{B}^{-1}([1:\pm 1: \pm 1: \pm 1]) \cong X^2$ to the one-dimensional manifold $\mathcal{A}^{-1}([1:\pm 1: \pm 1: \pm 1]) \cong X^1$.

Consider $\widetilde{s}|_{\mathcal{B}^{-1}([1: - 1: 1: 1])}: \mathcal{B}^{-1}([1: -
1: 1: 1]) \rightarrow \mathcal{A}^{-1}([1:  1: 1: 1])$.  
The fiber $\mathcal{B}^{-1}([1: - 1: 1: 1])$ is obtained by first blowing up
the point $[1:-1:1:1]$ and then blowing up the proper transforms of the lines
\begin{eqnarray}\label{lines1}
z_1 = -z_2 = z_3, \quad z_1 = -z_2 = z_4, \quad z_1 = z_3 = z_4, \quad -z_2 = z_3 = z_4.
\end{eqnarray}
Similarly, the fiber $\mathcal{A}^{-1}([1:  1: 1: 1])$ is obtained by first blowing up
the point $[1:1:1:1]$ and then blowing up the proper transforms of the lines
\begin{eqnarray}\label{lines2}
z_1 = z_2 = z_3, \quad z_1 = z_2 = z_4, \quad z_1 = z_3 = z_4, \quad z_2 = z_3 = z_4.
\end{eqnarray}
Consider just the point blow ups at $[1:-1:1:1]$ and $[1:1:1:1]$, respectively.  There are coordinates
\begin{eqnarray*}
([y_1:y_2:y_3:y_4],[m_1:m_2:m_3]) \quad \mbox{where} \quad  (m_1,m_2,m_3) \sim (y_1 - y_4,y_2+y_4,y_3-y_4),
\end{eqnarray*}
and
\begin{eqnarray*}
([x_1:x_2:x_3:x_4],[n_1:n_2:n_3]) \quad \mbox{where} \quad (n_1,n_2,n_3) \sim (x_1 - x_4,x_2-x_4,x_3-x_4) 
\end{eqnarray*}
in neighborhoods of $E_{[1:-1:1:1]}$ within $Y^3_1$ and of $E_{[1:1:1:1]}$
within $X^3_1$.  In fact, these serve as coordinates in a neighborhood of the
generic points of $\mathcal{B}^{-1}([1: - 1: 1: 1])$ and $\mathcal{A}^{-1}([1:
 1: 1: 1])$, within $Y^3$ and $X^3$, respectively.  (By generic, we mean
points that are not altered by the blow ups of the proper transforms of the lines (\ref{lines1}) and
(\ref{lines2}).)

In these coordinates, $\widetilde{s}|_{\mathcal{B}^{-1}([1: - 1: 1: 1])}:
\mathcal{B}^{-1}([1: - 1: 1: 1]) \rightarrow \mathcal{A}^{-1}([1:  1: 1: 1])$
is given by
\begin{eqnarray*}
\widetilde{s}|_{\mathcal{B}^{-1}([1: - 1: 1: 1])}([m_1:m_2:m_3]) = [m_1:m_2:m_3].
\end{eqnarray*}
At generic points of $\mathcal{A}^{-1}([1:  1: 1: 1])$, $Z$ is described
by the equation $n_1 = n_2$.   Therefore, $\mathscr{Z}_{-,+,+} = (\widetilde{s}|_{\mathcal{B}^{-1}([1: - 1: 1: 1])})^{-1}(Z)$ is described at generic points of $\mathcal{B}^{-1}([1: - 1: 1: 1])$
by 
\begin{eqnarray}\label{EQN:PULLBACK_Z}
m_1 = m_2.  
\end{eqnarray}

The fiber ${\mathcal{A}}^{-1}([1: - 1:1:1])$ is a result of blowing up  the
line $\{z_1 = z_3 = z_4\} \in A^1$.  
Coordinates in a neighborhood of this fiber are given by
\begin{eqnarray*}
([z_1:z_2:z_3:z_4],[p_1:p_2]), \quad \mbox{where} \quad (p_1,p_2) \sim (z_1-z_4,z_3-z_4).
\end{eqnarray*}
Notice that generic points of $Y^3$ (i.e. those not in $\mathcal{B}^{-1}(B^0
\cup B^1)$) can be described by the homogeneous coordinates $[y_1:y_2:y_3:y_4]$
on $\P^3$ and generic points on $X^3$ (those not in $\mathcal{A}^{-1}(A^0
\cup A^1)$) can be described by the homogeneous
coordinates $[z_1:z_2:z_3:z_4]$.  At these generic points
we have $\pr([y_1:y_2:y_3:y_4]) = [y_1:y_2:y_3:y_4]$.   Since $\pr$ is
continuous, this implies that at generic points of the fiber $\mathcal{B}^{-1}([1: - 1: 1: 1])$, we have
\begin{eqnarray*}
[p_1:p_2] = \pr([m_1:m_2:m_3]) = [m_1:m_3].
\end{eqnarray*}
In particular, Equation (\ref{EQN:PULLBACK_Z}) places no restriction on
the values $[p_1:p_2]$, so we conclude that $\dim(\pr(\mathscr{Z}_{-,+,+}))~=~1$
and hence that $\pr(\mathscr{Z}_{-,+,+}) = \mathcal{A}^{-1}([1:-1:1:1])$.

If we repeat the previous calculation over $[1:1:-1:1]$, the homogeneous coordinates
on $\mathcal{B}^{-1}([1:  1: -1: 1])$ are given by
\begin{eqnarray}
([y_1:y_2:y_3:y_4],[m_1:m_2:m_3]) \quad \mbox{where} \quad (m_1,m_2,m_3) \sim (y_1 - y_4,y_2-y_4,y_3+y_4),
\end{eqnarray}
and again, $(\widetilde{s} | _{{\mathcal{B}}^{-1}([1:  1: -1: 1])}^{_1}(Z)$ is given at generic points of ${\mathcal{B}}^{-1}([1:  1: -1: 1])$ by
\begin{eqnarray}\label{EQN:PULLBACK_Z2}
m_1 = m_2.
\end{eqnarray}

Meanwhile, the fiber ${\mathcal{A}}^{-1}([1:  1:-1:1])$ is a result of blowing
up  the line $\{z_1 = z_2 = z_4\} \in A^1$. 
Coordinates in a neighborhood of this fiber are given by
\begin{eqnarray*}
([z_1:z_2:z_3:z_4],[p_1:p_2]), \quad \mbox{where} \quad (p_1,p_2) \sim (z_1-z_4,z_2-z_
4).
\end{eqnarray*}
At generic points of the fiber $\mathcal{B}^{-1}([1: - 1: 1: 1])$, we have
\begin{eqnarray*}
[p_1:p_2] = \pr([m_1:m_2:m_3]) = [m_1:m_2].
\end{eqnarray*}
In particular, Equation (\ref{EQN:PULLBACK_Z}) places the restriction $[p_1:p_2] = [1:1]$ on points in $\pr(\mathscr{Z}_{+,-,+})$.  We conclude that $\dim(\pr(\mathscr{Z}_{+,-,+}))~=~0$.

Using very similar calculations, one finds that the component of
$\pr(\mathscr{Z}_{+,+,-})$ over $[1:1:1:-1]$ is a single point on
${\mathcal{A}}^{-1}([1:  1:1:-1])$ and that the component of
$\pr(\mathscr{Z}_{-,-,-})$ over $[1:-1:-1:-1]$ is the whole one-dimensional fiber
${\mathcal{A}}^{-1}([1:  -1:-1:-1])$.  

In summary, Proposition \ref{ai_prop} gives
\begin{eqnarray}
s^*[Z] = [Z] + [{\mathcal{A}}^{-1}([1:  -1:1:1])] + [{\mathcal{A}}^{-1}([1:  -1:-1:-1])].
\end{eqnarray}

It remains to project $[{\mathcal{A}}^{-1}([1:  -1:1:1])]$ and
$[{\mathcal{A}}^{-1}([1:  -1:-1:-1])]$ back into the basis ${\bm B}_2^3$.
The divisor $D^{S_{14}}$ is obtained as the blow up of the proper transform of $\{z_1=z_3=z_4\}$ within $X_1^3$.  It is biholomorphic to $\P^1 \times \P^1$, with the first factor parameterized by points of the line $z_1=z_3=z_4$ and the second factor parameterized by the fibers of the blow up.  In particular, any two fibers have cohomologous fundamental class.  Thus, 
\begin{eqnarray*}
[{\mathcal{A}}^{-1}([1:  -1:1:1])] \cong [D^{S_3} \cap D^{S_{14}}] = [Z_{14}],
\end{eqnarray*}
because $D^{S_3} \cap D^{S_{14}}$ is the fiber of the blow up over the intersection
of the proper transform of $\{z_1=z_3=z_4\}$ with $E_{[0:1:0:0]}$.
Using similar reasoning, we have
\begin{eqnarray*}
[{\mathcal{A}}^{-1}([1:  -1:-1:-1])] \cong [D^{S_4} \cap D^{S_{13}}] = [Z_{15}].
\end{eqnarray*}

We summarize our calculation with:
\begin{proposition}\label{PROP_PULLBACK_DIM3} With respect to the basis  ${\bm B}_2^3$, $s^*: H^{2,2}(X^3;\C) \rightarrow H^{2,2}(X^3;\C)$ is given by the matrix
{\footnotesize
\begin{eqnarray*}
s^* = \left[
\begin{array}{cccccccccccccccc}
4 & 0 & 0 & 0& 0& 0& 0& 0& 0& 0& 0& 0& 0& 0& 0& 0 \\
0 & 4 & 0 & 0& 0& 0& 0& 0& 0& 0& 0& 0& 0& 0& 0& 0 \\
0 & 0 & 4 & 0& 0& 0& 0& 0& 0& 0& 0& 0& 0& 0& 0& 0 \\
0 & 0 & 0 & 4& 0& 0& 0& 0& 0& 0& 0& 0& 0& 0& 0& 0 \\
0 & 0 & 0 & 0& 4& 0& 0& 0& 0& 0& 0& 0& 0& 0& 0& 0 \\
0 & 0 & 0 & 0& 0& 4& 0& 0& 0& 0& 0& 0& 0& 0& 0& 0 \\
0 & 0 & 0 & 0& 0& 0& 4& 0& 0& 0& 0& 0& 0& 0& 0& 0 \\
0 & 0 & 0 & 0& 0& 0& 0& 4& 0& 0& 0& 0& 0& 0& 0& 0 \\
0 & 0 & 0 & 0& 0& 0& 0& 0& 4& 0& 0& 0& 0& 0& 0& 0 \\
0 & 0 & 0 & 0& 0& 0& 0& 0& 0& 4& 0& 0& 0& 0& 0& 0 \\
0 & 0 & 0 & 0& 0& 0& 0& 0& 0& 0& 4& 0& 0& 0& 0& 0 \\
0 & 0 & 0 & 0& 0& 0& 0& 0& 0& 0& 0& 2& 0& 0& 0& 0 \\
0 & 0 & 0 & 0& 0& 0& 0& 0& 0& 0& 0& 0& 2& 0& 0& 0 \\
0 & 0 & 0 & 0& 0& 0& 0& 0& 0& 0& 0& 0& 0& 2& 0& 1 \\
0 & 0 & 0 & 0& 0& 0& 0& 0& 0& 0& 0& 0& 0& 0& 2& 1 \\
0 & 0 & 0 & 0& 0& 0& 0& 0& 0& 0& 0& 0& 0& 0& 0& 1 \\
\end{array}
\right]
\end{eqnarray*}
}
\end{proposition}

%bc_abc,   {23} {123}   {1456} {123} 
%bc_bcd, {23}{234}       {1456} {156} *
%bc_bce, {23} {235}    {1456} {146}  *
%bc_bcf, {23} {236}    {1456} {145}   *
%bd_abd, {24} {124}    {1356} {356} 
%bd_bde, {24} {245}    {1356} {136}  *
%bd_bdf, {24} {246}    {1356} {135}   *
%be_abe {25} {125}     {1346} {125} 
%be_bef, {25} {256}   { 1346} {134}    
%bf_abf, {26} {126}     {1345} {126}    *
%
%ab_cf, {12} {36}      {12} {1245}   Very bad guy
%
%ae_bc {15} {23}     {15} {1456} 
%,ae_bd, {15} {24}    {15} {1356} 
%ae_bf, {15} {26}       {15} {1345} 
%ad_be, {14} {25}      {14} {1346} 
%ae_bcf" {15} {236}     {15} { 145} 
%

% 
%\subsubsection{Case $N=4$, $k=2$}
%\subsubsection{Case $N=4$, $k=3$}
%

%% file: data.tex
\section{Dynamical Degree Data}\label{SEC:CATALOG}

Let $P=\{{\bm p_1},{\bm p_2},p_3,\ldots,p_n\}$ contain at least three points,
and recall $N:=n-3$. We distinguish the points ${\bm p_1}$ and ${\bm p_2}$ in
boldface, as these are precisely the distinguished points in our coordinate
system (see Section \ref{projective}). 
Using the bases from Section \ref{SEC:DYN_DEG}, we explicitly computed the following dynamical degrees:

\begin{itemize}
\item[$\quad$] $|P|=5$ (equivalently $N=2$), for all $\rho\in S_P$, we compute $\lambda_1(f_\rho)$, 
\item[$\quad$] $|P|=6$ (equivalently $N=3$), for all $\rho\in S_P$, we compute $\lambda_1(f_\rho)$, and $\lambda_2(f_\rho)$,
\item[$\quad$] $|P|=7$ (equivalently $N=4$), for all $\rho\in S_P$, we compute $\lambda_1(f_\rho)$, 
\item[$\quad$] $|P|=8$ (equivalently $N=5$), for all cyclic permutations $\rho\in S_P$, we compute $\lambda_1(f_\rho)$. 
\end{itemize}
They are presented in Tables \ref{TABLE1}-\ref{TABLE4}, below.
As previously mentioned, the methods employed in Section \ref{SEC:DYN_DEG} should generalize to computing the dynamical degrees $\lambda_k(f_\rho)$ for $f_\rho:X^N~\dashrightarrow~X^N$ for arbitrary $N$ and $1\leq k\leq N$. 

%Let $P=\{{\bm p_1},{\bm p_2},p_3,\ldots,p_n\}$ contain at least three points,
%and recall $N:=n-3$. We distinguish the points ${\bm p_1}$ and ${\bm p_2}$ in
%boldface, as these are precisely the distinguished points in our coordinate
%system (see Section \ref{projective}). We list some of the dynamical degrees
%that we computed in the tables below. 
%Any element of $\rho\in S_P$ can be
%written uniquely as a product of disjoint cycles. In the left-most column in
%each table, we use these cycles to encode the permutation $\rho\in S_P$. We
%have included the permutations $\rho\in S_P$ that yield maps
%$f_\rho:X^N\dashrightarrow X^N$ with the more interesting dynamical degrees. 

\begin{question}
To what extent does the structure of the permutation affect the dynamical degrees? There are some patterns that are evident in the tables below. For instance, when $\rho\in S_P$ consists of just one cycle of length $n$, the dynamical degrees are ``more complicated'' from an algebraic point of view (they tend to have higher algebraic degree). A somewhat related question concerns the characteristic polynomials: in almost all examples, the degree of the eigenvalue corresponding to $\lambda_k(f_\rho)$ is strictly less than the dimension of $H^{k,k}(X^N;\C)$, and the characteristic polynomial factors. What is the dynamical significance of i) the number of factors, and ii) the algebraic multiplicity of each factor?
\end{question}

\begin{remark}\label{REM:ONE_STABILITY} One can easily notice from the tables that the first dynamical
degree of $f_\rho: X^N \dashrightarrow X^N$ is always an algebraic integer of
degree $\leq N+3$ which is significantly less then ${\rm dim}(H^{1,1}(X^N;\C)) = 2^{N+2}-\binom{N+3}{2}-1$ when $N \geq 3$.  This can be explained as follows:  Let $Z^N$ be the blow up
of $\P^N$ at the $N+2$ points from $A^0$ and let $\pi: X^N \rightarrow Z^N$ be
the resulting blow down map.  One can check that $\pi \circ f_\rho \circ
\pi^{-1}: Z^N \dashrightarrow Z^N$ is $1$-stable so that its dynamical degree
is an algebraic integer of degree less than or equal to ${\rm
dim}(H^{1,1}(Z^N;\C)) = N+3$.  The result follows for $f_\rho$ since dynamical
degrees are invariant under birational conjugacy.  \end{remark}

\begin{table}[h!]
\begin{center}
\begin{tabular}{|l|}
\hline
 ${\bm p_1}\mapsto {p_3}\mapsto {\bm p_1}  \qquad     {\bm p_2}\mapsto p_4\mapsto p_5\mapsto {\bm p_2}$        \\
 $\lambda_1(f_\rho) \approx 2.2292085  \qquad   \lambda^4+\lambda^3-2\lambda^2-8\lambda-8$ \\
\hline 
 ${\bm p_1}\mapsto {p_3}\mapsto {\bm p_2}\mapsto p_4\mapsto p_5\mapsto {\bm p_1}$    \\
 $\lambda_1(f_\rho) \approx 2.2755888  \qquad \lambda^5-\lambda^4-8\lambda-16$ \\
\hline
 ${\bm p_1}\mapsto {\bm p_2}\mapsto p_3\mapsto p_4\mapsto p_5\mapsto {\bm p_1}$    \\
 $\lambda_1(f_\rho) \approx2.2667836 \qquad  \lambda^5-2\lambda^3-4\lambda^2-16$ \\
\hline
\end{tabular}
\end{center}
\caption{\label{TABLE1}Data for $f_\rho^*:H^{1,1}(X^2;\C)\to H^{1,1}(X^2;\C)$; the permutation $\rho\in S_P$ is given in terms of cycles, an approximate value of the dynamical degree $\lambda_1(f_\rho)$ is given as well as the minimal polynomial for $\lambda_1(f_\rho)$. There are $120$ such maps $f_\rho:X^2\dashrightarrow X^2$ corresponding to all permutations $\rho\in S_P$. The examples in this chart (and all maps which are birationally conjugate to any of these examples) are the only maps $f_\rho:X^2\dashrightarrow X^2$ for which $\lambda_1(f_\rho)\neq 2$.}
\vspace{0.15in}
\end{table}

%\begin{table}[h!]
%\begin{tabular}{|c|c|c|}
%\hline
%%$N \setminus k$ & 0 & 1 & 2 & 3 & 4 & 5 & 6  \\
%The permutation $\rho\in S_P$ & The minimal polynomial for $\lambda_1(f_\rho)$ & $\lambda_1(f_\rho)$\\
%${\bm p_1}\mapsto {p_3}\mapsto {\bm p_1}  \quad     {\bm p_2}\mapsto p_4\mapsto p_5\mapsto {\bm p_2}$    &        $\lambda^4+\lambda^3-2\lambda^2-8\lambda-8$ & $\approx2.2292085$ \\
%${\bm p_1}\mapsto {p_3}\mapsto {\bm p_2}\mapsto p_4\mapsto p_5\mapsto {\bm p_1}$   &      $\lambda^5-\lambda^4-8\lambda-16$   & $\approx 2.2755888$\\
%${\bm p_1}\mapsto {\bm p_2}\mapsto p_3\mapsto p_4\mapsto p_5\mapsto {\bm p_1}$   &      $\lambda^5-2\lambda^3-4\lambda^2-16$  & $\approx2.2667836$ \\
%\hline
%\end{tabular}
%\caption{Data for $f_\rho^*:H^{1,1}(X^2;\C)\to H^{1,1}(X^2;\C)$}
%\vspace{0.15in}
%\end{table}

\begin{table}
\begin{tabular}{|l|}
\hline
${\bm p_1}\mapsto {p_3}\mapsto {\bm p_1}  \qquad      {\bm p_2}\mapsto p_4\mapsto p_5\mapsto {\bm p_2} \quad p_6\mapsto p_6$       \\  $\lambda_1(f_\rho) \approx2.2292085 \qquad \lambda^4 + \lambda^3 -2\lambda^2 - 8\lambda - 8$  \\
$\lambda_2(f_\rho) \approx 4.4584171 \qquad \lambda^4+2\lambda^3-8\lambda^2-64\lambda-128$ \\
\hline

${\bm p_1}\mapsto {p_3}\mapsto {\bm p_2}\mapsto p_4\mapsto p_5\mapsto {\bm p_1}   \quad p_6\mapsto p_6$    \\
$\lambda_1(f_\rho) \approx 2.2755888 \qquad \lambda^5-\lambda^4-8\lambda-16$   \\
$\lambda_2(f_\rho) \approx 4.5511777 \qquad \lambda^5-2\lambda^4-128\lambda-512$ \\

\hline
${\bm p_1}\mapsto {\bm p_2}\mapsto p_3\mapsto p_4\mapsto p_5\mapsto {\bm p_1}  \quad p_6\mapsto p_6$  \\
$\lambda_1(f_\rho) \approx 2.2667836 \qquad \lambda^5-2\lambda^3-4\lambda^2-16$ \\
$\lambda_2(f_\rho) \approx 4.5335672 \qquad \lambda^5-8\lambda^3-32\lambda^2-512$ \\

\hline
${\bm p_1}\mapsto p_3\mapsto {\bm p_1}\quad {\bm p_2}\mapsto p_4\mapsto p_5\mapsto p_6\mapsto {\bm p_2}$  \\
$\lambda_1(f_\rho) \approx 2.3461556 \qquad \lambda^4-\lambda^3-4\lambda-8$  \\
$\lambda_2(f_\rho) \approx 4.6658733 \qquad \lambda^9-3\lambda^8-16\lambda^6-192\lambda^5 +384\lambda^4+128\lambda^3+6144\lambda-8192$  \\

\hline

${\bm p_1}\mapsto p_3\mapsto p_4\mapsto {\bm p_1}\quad {\bm p_2}\mapsto p_5\mapsto p_6\mapsto {\bm p_2}$  \\
$\lambda_1(f_\rho) \approx 2.4675038 \qquad \lambda^3-\lambda^2-2\lambda-4$\\
$\lambda_2(f_\rho) \approx 4.7900395 \qquad \lambda^6-\lambda^5-4\lambda^4-64\lambda^3-16\lambda^2-64\lambda+256$  \\
\hline

${\bm p_1}\mapsto {\bm p_2}\mapsto p_3\mapsto p_4\mapsto p_5\mapsto p_6\mapsto {\bm p_1}$  \\
$\lambda_1(f_\rho) \approx 2.4316847 \qquad \lambda^5+2\lambda^4+2\lambda^3-8\lambda^2+8\lambda-16$ \\
$\lambda_2(f_\rho) \approx 4.80241001\quad $ \footnotesize{$\lambda^{14}-4\lambda^{13}+12\lambda^{12}-88\lambda^{11}
+256\lambda^{10}-1152\lambda^9+768\lambda^8 -3072\lambda^7$} \\
\hspace{1.5in} \footnotesize{$+12288\lambda^6 +24576\lambda^5-65536\lambda^4 + 458752\lambda^3+1048576\lambda-4194304$} \\
\hline

${\bm p_1}\mapsto p_3\mapsto {\bm p_2}\mapsto p_4\mapsto p_5\mapsto p_6\mapsto {\bm p_1}$  \\
$\lambda_1(f_\rho) \approx 2.4576736 \qquad \lambda^6-\lambda^5-4\lambda^4-16\lambda -32$ \\
$\lambda_2(f_\rho) \approx 4.84568805 \qquad \lambda^{13} - 3\lambda^{12} - 16\lambda^{10}  - 256\lambda^8 
- 3584\lambda^7 + 7168\lambda^6$ \\
 \hspace{2in} $+ 2048\lambda^5
+ 32768\lambda^3 + 1572864\lambda - 2097152$ \\
\hline

${\bm p_1}\mapsto p_3\mapsto p_4\mapsto {\bm p_2}\mapsto p_5\mapsto p_6\mapsto {\bm p_1}$  \\
$\lambda_1(f_\rho) \approx 2.4675037 \qquad \lambda^3-\lambda^2-2\lambda-4$\\
$\lambda_2(f_\rho) \approx 4.7900395 \qquad \lambda^6 - \lambda^5 - 4\lambda^4 - 64\lambda^3 - 16\lambda^2 - 64\lambda + 25$ \\
\hline
\end{tabular}
\vspace{0.15in}
\caption{Data for $f_\rho^*:H^{1,1}(X^3;\C)\to H^{1,1}(X^3;\C)$ and $f_\rho^*:H^{2,2}(X^3;\C)\to H^{2,2}(X^3;\C)$; the permutation $\rho\in S_P$ is given in terms of cycles, an approximate value of the dynamical degrees $\lambda_1(f_\rho)$ and $\lambda_2(f_\rho)$ are given as well as the minimal polynomials for $\lambda_1(f_\rho)$ and $\lambda_2(f_\rho)$. There are $720$ such maps $f_\rho:X^3\dashrightarrow X^3$ corresponding to all permutations $\rho\in S_P$. The examples in this chart (and all maps which are birationally conjugate to any of these examples) are the only maps $f_\rho:X^3\dashrightarrow X^3$ for which $\lambda_1(f_\rho)\neq 2$, and they are also the only maps  $f_\rho:X^3\dashrightarrow X^3$ for which $\lambda_2(f_\rho)\neq 4$.}
\vspace{0.1in}
\end{table}

\begin{table}
\begin{center}
\begin{tabular}{|l|}
\hline
${\bm p_1}\mapsto {\bm p_2}\mapsto p_3\mapsto p_4\mapsto p_5\mapsto {\bm p_1}\quad p_6\mapsto p_6\quad p_7\mapsto p_7$  \\
$\lambda_1(f_\rho) \approx 2.2667836 \qquad \lambda^5-2\lambda^3-4\lambda^2-16$ \\
\hline
${\bm p_1}\mapsto {\bm p_2}\mapsto p_3\mapsto p_4\mapsto p_5\mapsto {\bm p_1}\quad p_6\mapsto p_7\mapsto p_6$  \\
$\lambda_1(f_\rho) \approx 2.2667836 \qquad \lambda^5-2\lambda^3-4\lambda^2-16$  \\
\hline
${\bm p_1}\mapsto p_3\mapsto {\bm p_2} \mapsto p_4\mapsto p_5\mapsto {\bm p_1}\quad p_6\mapsto p_6\quad p_7\mapsto p_7$  \\
$\lambda_1(f_\rho) \approx 2.2755888 \qquad \lambda^5-\lambda^4-8\lambda-16$ \\
\hline
${\bm p_1}\mapsto p_3\mapsto {\bm p_2} \mapsto p_4\mapsto p_5\mapsto {\bm p_1}\quad p_6\mapsto  p_7\mapsto p_6$  \\
$\lambda_1(f_\rho) \approx 2.2755888 \qquad \lambda^5-\lambda^4-8\lambda-16$ \\
\hline
${\bm p_1}\mapsto p_3\mapsto {\bm p_1}\quad {\bm p_2}\mapsto p_4\mapsto p_5\mapsto {\bm p_2}\quad p_6\mapsto p_6\quad p_7\mapsto p_7$  \\
$\lambda_1(f_\rho) \approx 2.2292085 \qquad \lambda^4+\lambda^3-2\lambda^2-8\lambda-8$ \\
\hline
${\bm p_1}\mapsto p_3\mapsto {\bm p_1}\quad {\bm p_2}\mapsto p_4\mapsto p_5\mapsto {\bm p_2}\quad p_6\mapsto p_7\mapsto p_6$  \\
$\lambda_1(f_\rho) \approx 2.2292085 \qquad \lambda^4+\lambda^3-2\lambda^2-8\lambda-8$ \\
\hline
${\bm p_1}\mapsto {\bm p_2}\mapsto p_3\mapsto p_4\mapsto p_5\mapsto p_6\mapsto {\bm p_1} \quad p_7\mapsto p_7$  \\
$\lambda_1(f_\rho) \approx 2.4316847 \qquad \lambda^5-2\lambda^4+2\lambda^3-8\lambda^2+8\lambda-16$ \\
\hline
${\bm p_1}\mapsto p_3\mapsto {\bm p_2}\mapsto p_4\mapsto p_5\mapsto p_6\mapsto {\bm p_1}\quad p_7\mapsto p_7$  \\
$\lambda_1(f_\rho) \approx 2.4576736 \qquad \lambda^6-\lambda^5-4\lambda^3-16\lambda-32$ \\
\hline
${\bm p_1}\mapsto p_3\mapsto p_4\mapsto {\bm p_2}\mapsto p_5\mapsto p_6\mapsto {\bm p_1}\quad p_7\mapsto p_7$  \\
$\lambda_1(f_\rho) \approx 2.4675038 \qquad \lambda^3-\lambda^2-2\lambda -4$ \\
\hline
${\bm p_1}\mapsto p_2\mapsto p_3\mapsto p_4\mapsto p_5\mapsto p_6 \mapsto p_7\mapsto {\bm p_1}$  \\
$\lambda_1(f_\rho) \approx 2.5339057 \qquad \lambda^7-2\lambda^5-4\lambda^4-8\lambda^3-16\lambda^2-64$ \\
\hline
${\bm p_1}\mapsto p_3\mapsto {\bm p_2}\mapsto p_4\mapsto p_5\mapsto p_6\mapsto p_7\mapsto {\bm p_1}$  \\
$\lambda_1(f_\rho) \approx 2.5746797 \qquad \lambda^7-\lambda^6-4\lambda^4-8\lambda^3-32\lambda-64$ \\
\hline
${\bm p_1}\mapsto p_3\mapsto p_4\mapsto {\bm p_2}\mapsto p_5\mapsto p_6\mapsto p_7\mapsto {\bm p_1}$  \\
$\lambda_1(f_\rho) \approx 2.5985401 \qquad \lambda^7-\lambda^6-2\lambda^5-16\lambda^2-32\lambda -64$ \\
\hline
\end{tabular}
\end{center}
%\vspace{0.05in}
\caption{Data for $f_\rho^*:H^{1,1}(X^4;\C)\to H^{1,1}(X^4;\C)$; the permutation $\rho\in S_P$ is given in terms of cycles, an approximate value of the dynamical degree $\lambda_1(f_\rho)$ is given as well as the minimal polynomial for $\lambda_1(f_\rho)$. There are $5040$ such maps $f_\rho:X^4\dashrightarrow X^4$ corresponding to all permutations $\rho\in S_P$. The examples in this chart (and all maps which are birationally conjugate to any of these examples) are the only maps $f_\rho:X^4\dashrightarrow X^4$ for which $\lambda_1(f_\rho)\neq 2$.}
\end{table}

\begin{table}
\begin{center}
\begin{tabular}{|l|}
\hline
${\bm p_1}\mapsto {\bm p_2}\mapsto p_3\mapsto p_4\mapsto p_5\mapsto p_6\mapsto p_7\mapsto p_8\mapsto {\bm p_1}$  \\ 
$\lambda_1(f_\rho) \approx 2.5986551 \qquad \lambda^7-2\lambda^6+2\lambda^5-8\lambda^4+8\lambda^3-32\lambda^2+32\lambda-64$ \\
\hline
${\bm p_1}\mapsto p_3\mapsto {\bm p_2}\mapsto p_4\mapsto p_5\mapsto p_6\mapsto p_7\mapsto p_8\mapsto {\bm p_1}$  \\
$\lambda_1(f_\rho) \approx 2.6494359 \qquad \lambda^3-4\lambda-8$\\
\hline
${\bm p_1}\mapsto p_3\mapsto p_4\mapsto {\bm p_2} \mapsto {p_5}\mapsto p_6\mapsto p_7\mapsto p_8\mapsto {\bm p_1}$  \\
$\lambda_1(f_\rho) \approx 2.68518317 \qquad \lambda^7-3\lambda^6+4\lambda^5-8\lambda^4+8\lambda^3-16\lambda^2-52$ \\
\hline
${\bm p_1}\mapsto p_3\mapsto p_4\mapsto p_5\mapsto {\bm p_2}\mapsto p_6\mapsto p_7\mapsto p_8\mapsto {\bm p_1}$  \\
$\lambda_1(f_\rho) \approx 2.6980689 \qquad \lambda^4-\lambda^3-2\lambda^2-4\lambda-8$ \\
\hline
\end{tabular}
\end{center}
%\vspace{0.05in}
\caption{\label{TABLE4}Data for $f_\rho^*:H^{1,1}(X^5;\C)\to H^{1,1}(X^5;\C)$; the permutation $\rho\in S_P$ is given in terms of cycles, an approximate value of the dynamical degree $\lambda_1(f_\rho)$ is given as well as the minimal polynomial for $\lambda_1(f_\rho)$. Of the $40,320$ such maps $f_\rho:X^5\dashrightarrow X^5$ corresponding to all permutations $\rho\in S_P$, we present the data for representative examples corresponding to the cyclic permutations $\rho\in S_P$.} 
\end{table}